\documentclass[10pt,a4paper]{amsart}
\usepackage[utf8]{inputenc}
\usepackage{amsmath}
\usepackage{amsfonts}
\usepackage{amssymb}
\usepackage{graphicx}
\usepackage[all]{xy}
\usepackage[left=2cm,right=2cm,top=2cm,bottom=2cm]{geometry}
\usepackage[color=blue!30,textsize=tiny]{todonotes}
\usepackage{tikz}
\usepackage{xcolor}
\usetikzlibrary{arrows}

\author{Geoffrey Powell}
\title[Lie algebra homology with tensor coefficients]{Lie algebra homology  with coefficients tensor products of the adjoint representation in relative polynomial degree $2$}
\address{Univ Angers, CNRS, LAREMA, SFR MATHSTIC, F-49000 Angers, France}
\email{Geoffrey.Powell@math.cnrs.fr}
\urladdr{https://math.univ-angers.fr/~powell/}


\newtheorem{THM}{Theorem}

\newtheorem{COR}{Corollary}
\newtheorem{thm}{Theorem}[section]
\newtheorem{prop}[thm]{Proposition}

\newtheorem{cor}[thm]{Corollary}
\newtheorem{lem}[thm]{Lemma}
\theoremstyle{definition}
\newtheorem{defn}[thm]{Definition}
\newtheorem{exam}[thm]{Example}

\theoremstyle{remark}
\newtheorem{rem}[thm]{Remark}
\newtheorem{REM}{Remark}
\newtheorem{nota}[thm]{Notation}

\renewcommand{\hom}{\mathrm{Hom}}
\newcommand{\g}{\mathfrak{g}}
\renewcommand{\phi}{\varphi}
\newcommand{\nat}{\mathbb{N}}
\newcommand{\rat}{\mathbb{Q}}
\newcommand{\lie}{\mathbb{L}\mathrm{ie}}
\newcommand{\cyclie}{\mathbb{C}\mathrm{ycLie}}
\newcommand{\liemod}{\mathrm{Lie}}

\newcommand{\schur}{\mathbf{S}}
\newcommand{\sym}{\mathfrak{S}}
\newcommand{\sgn}{\mathrm{sgn}}
\newcommand{\triv}{\mathrm{triv}}
\newcommand{\dbar}{\overline{\delta}}

\newcommand{\ob}{\mathrm{Ob\ }}
\newcommand{\fmodq}{\mathbf{mod}_\rat}
\newcommand{\adbar}{\overline{\mathrm{ad}}}
\newcommand{\coker}{\mathrm{coker}}
\newcommand{\id}{\mathrm{Id}}
\newcommand{\diff}{\mathsf{d}}

\newcommand{\f}{\mathcal{F}}
\newcommand{\fout}{\f^\mathrm{Out}}
\newcommand{\gr}{\mathbf{gr}}
\newcommand{\op}{^\mathrm{op}}
\newcommand{\cat}{\mathsf{Cat}}
\newcommand{\lieopd}{\mathrm{Lie}}

\begin{document}

\begin{abstract}
The homology of free Lie algebras with coefficients in tensor products of the adjoint representation working over $\rat$ contains important information on the homological properties of polynomial outer functors on free groups. The latter category was introduced in joint work with Vespa, motivated by the study of higher Hochschild homology of wedges of circles. 

There is a splitting of this homology by polynomial degree (for polynomiality with respect to the generators of the free Lie algebra) and one can consider the polynomial degree relative to the number of tensor factors in the coefficients. It suffices to consider the Lie algebra homology in homological degree one; this  vanishes in relative degree $0$ and is readily calculated in relative degree $1$.

This paper calculates the homology in relative degree $2$, which presents interesting features. This confirms a conjecture of Gadish and Hainaut.  
\end{abstract}

\maketitle

\subsection*{Keywords}
Lie algebra homology; adjoint representation; polynomial functor.

\section{Introduction}
\label{sect:intro}

For any Lie algebra $\g$ (working here over $\rat$) and natural number $r$, one can form the Lie algebra homology $H_* (\g; \g^{\otimes r})$ of $\g$ with coefficients in the $r$th tensor product of the adjoint representation. The case $r=0$ corresponds to trivial coefficients $\rat$; the case $r=1$ is homology with coefficients in the adjoint representation. Since the action of the symmetric group $\sym_r$ by place permutations on $\g^{\otimes r}$ is by $\g$-module isomorphisms, $H_* (\g; \g^{\otimes r})$ inherits a $\sym_r$-module structure. 

Here we are interested in the case where $\g$ is the free Lie algebra $\lie (V)$ on a finite-dimensional $\rat$-vector space $V$, considered as a functor of $V$. This is motivated by the relationship with higher Hochschild homology arising from the author's joint work with Christine Vespa \cite{2018arXiv180207574P}. More explicitly, in that paper we established the relationship with the category of {\em outer functors}, namely functors from $\mathbf{gr}$ (the category of finitely-generated free groups) to $\rat$-vector spaces on which inner automorphisms of free groups act trivially. In \cite{MR4696223}, using the infinitesimal interpretation of polynomial functors on free groups given in \cite{MR4835394}, this was related to Lie algebra homology with coefficients in tensor products of the adjoint representation.
This is also related to work of Nir Gadish and Louis Hainaut \cite{MR4799912}, who study the homology of certain configuration spaces and relate this to higher Hochschild homology; their calculations are used by Hainaut in  \cite{2023arXiv231116881H} to study homological properties of the category of polynomial outer functors. (See Appendix \ref{sect:motivation}, which is independent of the body of the text, for slightly more detail on the above.)

Working with the free Lie algebra,  since $H_* (\lie(V); \lie (V)^{\otimes r} )$ vanishes for $*>1$, its calculation reduces to that  of the $\sym_r$-modules
\begin{eqnarray*}
&&H_0 (\lie(V); \lie (V)^{\otimes r} )  \\
&&H_1 (\lie(V); \lie (V)^{\otimes r} ),
\end{eqnarray*}
naturally with respect to $V$. For $r=1$, the $H_1$-term is given by the Schur functor associated to the cyclic Lie operad (see Example \ref{exam:cyclie}):
\[
H_1 (\lie(V); \lie (V) ) \cong \cyclie (V).
\]

 One can restrict to studying $H_1 (\lie(V); \lie (V)^{\otimes r} )$, since the homological degree zero term can be related to this by using the isomorphism
 $$
 H_0 (\lie(V); \lie (V)^{\otimes r} )
\ \oplus \ 
(V \otimes \lie(V)^{\otimes r})
 \ \cong \  
 H_1 (\lie(V); \lie (V)^{\otimes r} )
\ \oplus \ 
 \lie(V)^{\otimes r} 
 $$
derived from the complex calculating the homology. This is advantageous, as observed by Gadish and Hainaut, since $H_1 (\lie(V); \lie (V)^{\otimes r} )$ is  usually smaller than $H_0 (\lie(V); \lie (V)^{\otimes r} )$.

The functor $V \mapsto H_1 (\lie(V); \lie (V)^{\otimes r} )$ is  analytic (in the sense of Joyal \cite{MR927763}).  In particular, for each $n \in \nat$, one can focus on the $n$th homogeneous component,  denoted here by $(-)^{[n]}$. The identification for $n \leq r+1$ is  straightforward (see Proposition \ref{prop:small_n}); there are $\sym_r$-equivariant natural isomorphisms
\[
H_1 (\lie (V); \lie(V)^{\otimes r})^{[n]} = \left\{ 
\begin{array}{ll}
0 & n \leq r \\
\Gamma^{r+1} (V) \boxtimes \triv_r & n = r+1,
\end{array}
\right.
\]  
where $\triv_r$ is the trivial representation of $\sym_r$ and $\Gamma^{r+1}(V) = (V^{\otimes r+1})^{\sym_{r+1}}$ is the $(r+1)$st divided power functor.

The purpose of this paper is to treat the next case, $n=r+2$. In the following statement, $\sgn_r$ denotes the sign representation of $\sym_r$, $\Lambda^{r+2}$ is the $(r+2)$nd exterior power functor and $\schur_{(r,1^2)}$ is the Schur functor associated to the partition $(r,1^2)$.

\begin{THM}
\label{THM}
For $r \in \nat$, there are $\sym_r$-equivariant natural isomorphisms
\[
H_1 (\lie (V); \lie(V)^{\otimes r})^{[r+2]}
\cong 
\left\{ 
\begin{array}{ll}
0 & r=0\\
\Lambda^3 (V) \boxtimes \sgn_1 & r=1 \\
\Lambda^{r+2} (V) \boxtimes \sgn_r \   \oplus \  \schur_{(r, 1^2)} (V) \boxtimes \triv_r
& r>1.
\end{array}
\right.
\]
\end{THM}

\begin{REM}
\label{REM:bimodule}
Using the Schur correspondence between representations of symmetric groups and polynomial functors, the result of Theorem \ref{THM} can be expressed for $r>1$ as stating that the Lie algebra homology is encoded by the representation of $\sym_{r+2} \times \sym_r$:
\begin{eqnarray}
\label{eqn:bimodule_THM1}
S_{(1^{r+2})} \boxtimes S_{(1^r)} \ \oplus \ S_{(r,1^2)} \boxtimes S_{(r)}.
\end{eqnarray}
Here, for $\lambda$ a partition of $n\in \nat$, $S_\lambda$ denotes the corresponding simple representation of $\sym_n$. With our conventions, $S_{(n)}$ is isomorphic to $\triv_n$ and $S_{(1^n)}$ to $\sgn_n$.
\end{REM}

Theorem \ref{THM} is straightforward for $r\in \{0, 1, 2\}$. The case  $r=3$ is a technical calculation that is postponed until Section \ref{sect:case_r3}. (This case was already known by computer calculation, for example as a consequence of calculations in \cite{MR4799912}; here it is established by hand.) These cases serve as the initial steps for the inductive proof that is given in Section \ref{sect:main}. 

A consequence of Theorem \ref{THM} is:

\begin{COR}
\label{COR:GH}
Gadish and Hainaut's conjecture \cite[Conjecture 6.7]{MR4799912} holds.
\end{COR}

Let us outline how Corollary \ref{COR:GH} follows from Theorem \ref{THM}. 
Gadish and Hainaut's conjecture is stated in terms of a bimodule that they denote $\Phi^1 [r+2, r]$ (using the indexing $r$ compatible with that of Theorem \ref{THM}). Their conjecture asserts that there is an isomorphism of $\sym_{r+2}\times \sym_r$-modules
\begin{eqnarray}
\label{eqn:bimodule_Phi1}
\Phi^1 [r+2,r] 
\cong
S_{(r+2)} \boxtimes S_{(r)} 
\ \oplus \ 
S_{(3,1^{r-1})} \boxtimes S_{(1^r)}.
\end{eqnarray}
 Observe that the right hand side of (\ref{eqn:bimodule_Phi1}) is obtained from (\ref{eqn:bimodule_THM1}) by replacing each occurrence of $S_\lambda$ for a partition $\lambda$ by $S_{\lambda^\dagger}$, where $\lambda^\dagger$ is the conjugate partition to $\lambda$.

In \cite[Section 4.4]{MR4799912}, Gadish and Hainaut explain how their objects $\Phi [-,-]$ are related to higher Hochschild homology, as considered in \cite{2018arXiv180207574P} and by Turchin and Willwacher in \cite{MR3982870}. Now, \cite[Theorem 6]{2018arXiv180207574P} gives a description of  higher Hochschild homology in terms of  algebraic objects defined in terms of functors on the category $\gr$. (This result also explains the passage between $\lambda$ and $\lambda^\dagger$ for the action of $\sym_{r+2}$ in the case at hand.)

These algebraic objects then relate to the Lie algebra homology using the results of \cite{MR4835394} and \cite{MR4696223} (see Appendix \ref{sect:motivation} for an outline of the main ideas). Indeed, the algebraic objects appearing in \cite[Theorem 6]{2018arXiv180207574P} are essentially determined by the family of functors $\omega \beta _s \rat \sym_s$, for $s\in \nat$, that appear in \cite[Theorem 7]{2018arXiv180207574P}. By that theorem, these form a family of injective cogenerators in the appropriate category of outer polynomial functors on $\gr$. After dualizing (so as to pass to functors on $\gr\op$), this corresponds to the property exhibited in Example \ref{exam:H_0_projective}, thereby establishing the link with Lie algebra homology. 

Putting the above together yields the explicit relationship between $\Phi^1 [r+2,r] $ and $H_1 (\lie (V); \lie(V)^{\otimes r})^{[r+2]}$ 
 which shows that Theorem \ref{THM} implies Corollary \ref{COR:GH}.

\subsection*{Acknowledgement} This project grew out of joint work with Christine Vespa; the author is very grateful  for her ongoing interest. The author is also grateful to Louis Hainaut for his interest and his comments on a previous version. 

The author thanks all the  anonymous referees for their careful reading of the paper and, in particular,  for their valuable suggestions, which have lead to an improvement of the exposition.

\section{Conventions and background} 
\label{sect:prelim}

\subsection{Representations of the symmetric groups and Schur functors}

We work over the rational numbers, $\rat$. The category of finite-dimensional $\rat$-vector spaces is denoted by $\fmodq$;  all undecorated tensor products $\otimes$ are understood to be taken over $\rat$. For $n \in \nat$, the symmetric group on $n$ letters is denoted $\sym_n$.
The category of $\rat \sym_n$-modules (or $\sym_n$-modules) is semisimple. The isomorphism classes of simple $\sym_n$-modules are indexed by the partitions $\lambda$ of $n$ (written $\lambda \vdash n$), with $S_\lambda$ denoting the associated simple module. The indexing convention is chosen so that $S_{(n)}$ is the trivial representation $\triv_n$ and $S_{(1^n)}$ the sign representation $\sgn_n$. For groups $H \subset G$, the notation $\downarrow^G_H$ denotes restriction from $G$-modules to $H$-modules and $\uparrow_H^G$ denotes induction.
  
The Schur functor construction establishes the correspondence between $\sym_n$-modules and functors from $\fmodq$ to $\rat$-vector spaces. For $M$ a $\sym_n$-module, the associated Schur functor is 
\[
V \mapsto V^{\otimes n}\otimes_{\sym_n} M,
\]
for $V \in \fmodq$, where $\sym_n$ acts by place permutations on $V^{\otimes n}$. Writing $T^n$ for the $n$th tensor product functor $V \mapsto V^{\otimes n}$ equipped with the place permutation action of $\sym_n$ (on the right), the Schur functor can be written as 
\[
T^n \otimes_{\sym_n} M.
\]
This is a polynomial functor that is homogeneous of degree $n$. (See \cite[Appendix I.A]{MR3443860} for polynomial functors in characteristic zero.) The Schur functor construction establishes an equivalence of categories between $ \sym_n$-modules and homogeneous polynomial functors of degree $n$; given such a polynomial functor $G$, the associated $\sym_n$-representation is given by 
the natural transformations
\[
\mathrm{Nat} ( T^n, G),
\]
where the $\sym_n$-action is given by the action on $T^n$. This can also be interpreted as a {\em cross-effect} functor, corresponding to passage to multilinear terms.

There is more structure: given two polynomial functors $F$ and $G$, one can form their pointwise tensor product, defined by $(F \otimes G) (V):= F(V) \otimes G(V)$. If $F$ and $G$ are homogeneous polynomial of degrees $m$ and $n$ respectively, then $F \otimes G$ is homogeneous polynomial of degree $m+n$. At the level of representations of the symmetric groups, this corresponds to the functor which sends a $\rat \sym_m$-module $M$ and a $\rat \sym_n$-module $N$ to the $\rat \sym_{m+n}$-module $(M \boxtimes N)\uparrow_{\sym _m \times \sym_n}^{\sym_{m+n}}$.

\begin{nota}
\label{nota:schur_functor}
For a partition $\lambda \vdash n$, the associated Schur functor is denoted $\schur_{\lambda}$, so that 
$\schur_\lambda (V) = V^{\otimes n}\otimes _{\sym_n} S_\lambda$.  
\end{nota}

\begin{exam}
For $\lambda = (n)$, the Schur functor $\schur_{(n)}$ identifies as the $n$th symmetric power functor $V \mapsto S^n (V) = (V^{\otimes n}) / \sym_n$. Over $\rat$, this is isomorphic to the $n$th divided power functor $V \mapsto \Gamma^n (V)= (V^{\otimes n})^{\sym_n}$.  There is a canonical surjection $V^{\otimes n} \twoheadrightarrow S^n (V)$ and a canonical inclusion $\Gamma^n (V) \hookrightarrow V^{\otimes n}$. 
 For $\lambda = (1^n)$, the Schur functor $\schur_{(1^n)}$ identifies as the $n$th exterior power functor $V \mapsto \Lambda^n (V)$. 

The $\nat$-graded functors $S^*$, $\Gamma^*$ and $\Lambda^*$ are all exponential functors (paying attention to signs in the case of $\Lambda^*$). In particular, for $a, b \in \nat$ and  $n = a+b$, there are natural  coproducts 
$\Lambda^n  \rightarrow \Lambda ^a  \otimes \Lambda^b $ and $\Gamma^n  \rightarrow \Gamma ^a  \otimes \Gamma^b$, and 
  natural products
$\Lambda ^a  \otimes \Lambda^b  \rightarrow \Lambda^n $ and 
$S^a  \otimes S^b  \rightarrow S^n$. 

To make these structures more explicit, we consider the evaluation on $V$ a vector space of dimension $t$, with chosen basis $\{x_1, \ldots , x_t\}$. Then $S^* (V)$ identifies as the polynomial algebra $\rat [x_1, \ldots , x_t]$;  $S^d (V)$ corresponds to the subspace of homogeneous polynomials of degree $d$. Then the products $S^a \otimes S^b \rightarrow S^{a+b}$ (evaluated on $V$) are given by the product in this polynomial algebra. The coproducts (evaluated on $V$) are given by equipping the polynomial algebra with the primitively-generated Hopf algebra structure with coproduct $\Delta$  determined by $\Delta x_i = x_1 \otimes 1 + 1 \otimes x_i$. Thus, for example, the coproduct $S^2 \rightarrow S^1 \otimes S^1$ is determined by $v w \mapsto v \otimes w + w \otimes v$, for any $v, w \in V$.

The exterior power functors can be treated by the same argument, by introducing a grading and using Koszul signs. Namely, we form the free graded commutative algebra on $sV$ (where $sV$ is $V$ placed in degree one) and equip this with the primitively-generated Hopf algebra structure. The underlying object can be identified with the exterior algebra on $x_1, \ldots , x_t$, where the grading corresponds to the polynomial degree.  The product (evaluated on $V$) then identifies as the wedge product $\wedge$. The coproduct is identified as in the case of the symmetric powers;  for example, the coproduct (evaluated on $V$) $\Lambda^2 \rightarrow \Lambda^1 \otimes \Lambda^1$ is given by $v \wedge w \mapsto v \otimes w - w \otimes v$.  
\end{exam}

\begin{exam}
\label{exam:LR_Pieri}
For two partitions $\lambda \vdash d$ and $\mu \vdash e $, the tensor product of the associated Schur functors $\schur _\lambda$ and $\schur_\mu$ is given by the Littlewood-Richardson rule. Namely, there is an isomorphism 
$$
\schur_\lambda \otimes \schur_\mu \cong \bigoplus_{\nu \vdash d+e} \schur_\nu^{\bigoplus c^\nu_{\lambda\mu}},
$$
where $c^\nu_{\lambda \mu}$ are the Littlewood-Richardson coefficients. Using the properties of the Schur functor construction recalled above, this follows from the classical Littlewood-Richardson rule for representations of the symmetric groups 
$$
(S_\lambda \boxtimes S_\mu)\uparrow_{\sym_d \times \sym_e}^{\sym_{d+e}}
\cong 
 \bigoplus_{\nu \vdash d+e} S_\nu^{\bigoplus c^\nu_{\lambda\mu}}.
$$

Specializing to the case where $\mu = (e)$, so that $\schur_\mu$ is the symmetric power $S^e$, this gives the {\em Pieri formula}:
$$
\schur_\lambda \otimes S^e \cong \bigoplus_{\substack{\nu \vdash d+e \\ \nu/ \lambda \in \mathrm{HS}}} \schur_\nu.
$$
Here, the condition $ \nu/ \lambda \in \mathrm{HS}$ states that the skew partition $\nu/ \lambda$ is a horizontal strip (using the convention that the Young diagram associated to a partition $\lambda$ has $\lambda_i$ boxes in the $i$th row). The horizontal strip condition means that the Young diagram of $\nu$ is obtained from that of $\lambda$ by adding at most one box in each column. (See \cite[Corollary 2.3.5]{MR1988690}, for example, paying attention to the fact that Weyman uses a different indexing convention.)

Taking $e=1$, this reduces to the classical Pieri rule: the sum is over all partitions $\nu \vdash d+1$ such that the Young diagram of $\nu$ contains that of $\lambda$.

There is an (equivalent) formulation of the Pieri formula for the case $\mu = (1^e)$, so that $\schur_\nu$ identifies as the exterior power functor $\Lambda^e$. In this case:
$$
\schur_\lambda \otimes \Lambda ^e \cong \bigoplus_{\substack{\nu \vdash d+e \\ \nu/ \lambda \in \mathrm{VS}}} \schur_\nu, 
$$
where $\mathrm{VS}$ now refers to vertical strips (exchanging the r\^ole of horizontal and vertical). 
\end{exam}

We record the following basic fact, which plays the r\^ole of Schur's lemma:

\begin{prop}
\label{prop:nat_schur}
For partitions $\lambda \vdash m$, $\mu \vdash n$,  
\begin{enumerate}
\item 
$\schur_\lambda$ is a simple functor from $\fmodq$ to $\rat$-vector spaces; 
\item 
$\mathrm{Nat} (\schur_\lambda, \schur_\mu ) \cong
\left\{ 
\begin{array}{ll}
 \rat & \lambda = \mu \\
 0 & \mbox{otherwise.}
\end{array}
\right.
$
\end{enumerate}
\end{prop}

For a partition $\lambda$, $l(\lambda)$ denotes its length, i.e., the number of non-zero terms in the partition. One has the following standard result on the connectivity of the simple Schur functors, which is a constituent of Schur-Weyl duality:

\begin{lem}
\label{lem:conn_schur}
For a partition $\lambda$,  $\schur_\lambda (V)$ is zero if and only if $\dim V < l(\lambda)$.
\end{lem}

\begin{proof}
(For completeness, we sketch a proof.) There is an explicit description of the Schur functors (see, for example, \cite[Chapter 2]{MR1988690}, noting that Weyman works over a general commutative ring and uses a different indexing convention). This implies, in particular,  that $\schur_\lambda $ is a subfunctor of  $\bigotimes_{j=1}^{\lambda_1} \Lambda ^{\mu_j}$, where $\mu := \lambda^\dagger$ is the conjugate partition, which has length $l (\mu) = \lambda_1$ (correspondingly, $\mu_1 = l (\lambda)$). The vanishing statement follows from the fact, for $t \in \nat$,  $\Lambda^t (V) $ is zero if and only if $\dim V< t$. 

That $\schur _\lambda (V)$ is non-zero for $\dim V = l (\lambda)$ follows from further analysis of the construction of $\schur_\lambda$. For example, this follows from \cite[Proposition 2.1.4]{MR1988690}.    
\end{proof}

We will also be working with functors from $\fmodq$ to $\rat$-vector spaces that are equipped with a $\sym_r$-action (for some fixed $r \in \nat$) and $\sym_r$-equivariant natural transformations. A basic example of such a functor is given by forming the tensor product of a homogeneous polynomial functor $F$ of degree $n \in \nat$ with a $\sym_r$-module $M$. This will be denoted by 
$
F\boxtimes M$.

We also work with  analytic functors (in the sense of \cite[Chapitre 4]{MR927763}). By definition, these are direct sums of homogeneous polynomial functors (not necessarily of the same degree).

\begin{nota}
\label{nota:components_analytic}
For $V \mapsto F(V)$ an analytic functor on $\fmodq$ and $n \in \nat$, write $F^{[n]}$ for the homogeneous component of degree $n$, so that one has the natural isomorphism
\[
 F(V) \cong \bigoplus_{n \in \nat} F^{[n]} (V).
\]
Thus, written without the variable $V$, $F \cong \bigoplus_{n \in \nat} F^{[n]}$. (We will also write $F(V)^{[n]}$ for $ F^{[n]} (V)$.) 
\end{nota}
 
\begin{exam}
\label{exam:lie_small_degree}
Consider the free Lie algebra functor $V \mapsto \lie (V)$, for $V \in \fmodq$. The underlying functor to $\rat$-vector spaces is analytic and one has the natural splitting
\[
\lie (V) \cong \bigoplus_{n \in \nat} \lie^{[n]} (V).
\]
Here $ \lie^{[n]} (V)$ identifies as  the subspace of $\lie (V)$  generated (as a vector space) by iterated brackets of length exactly $n$. 
In particular, $\lie^{[0]} (V) =0$, $\lie^{[1]}(V) = V$, $\lie ^{[2]} (V) \cong \Lambda^2 (V)$ and $\lie^{[3]} (V) \cong \schur_{(2,1)}(V)$. 
\end{exam}

\subsection{A form of the  de Rham differential}
\label{subsect:dR}

For later use, we introduce  $\diff$, a form of the de Rham differential written using the divided power functors. For $a, b \in \nat$, with $b>0$, $\diff$ is the natural transformation 
\begin{eqnarray}
\label{eqn:diff}
\diff \colon
\Lambda^a \otimes \Gamma^b 
\rightarrow 
\Lambda^{a+1} \otimes \Gamma^{b-1} 
\end{eqnarray}
that is given by the composite 
\[
\Lambda^a \otimes \Gamma^b  
\rightarrow 
\Lambda^a  \otimes \Gamma^1 \otimes \Gamma^{b-1}  
\cong 
\Lambda^a  \otimes \Lambda^1 \otimes \Gamma^{b-1}  
\rightarrow 
\Lambda^{a+1} \otimes \Gamma^{b-1} 
\]
in which the first map is the coproduct of $\Gamma^*$ (which corresponds to the natural inclusion $(V^{\otimes b})^{\sym_b} \subset V \otimes (V^{\otimes b-1})^{\sym_{b-1}}$) and the second is the product of $\Lambda^*$.

One has the following characterization of $\diff$:

\begin{lem}
\label{lem:diff_characterization}
For $a,b \in \nat$ with $b>0$, $\diff$ is the unique map which fits into the commutative diagram of natural transformations
\begin{eqnarray}
\label{eqn:characterize_diff}
\xymatrix{
\Lambda^a \otimes \Gamma^b 
\ar[r]^(.4)\diff
\ar@{^(->}[d]
&
\Lambda^{a+1}  \otimes \Gamma^{b-1} 
\ar@{^(->}[d]
\\
\Lambda^a  \otimes T^{b}
\ar[r]
&
\Lambda^{a+1} \otimes T^{ b-1},
}
\end{eqnarray}  
where the bottom map is induced by the product $\mu : \Lambda^a  \otimes \Lambda^1 \rightarrow \Lambda^{a+1} $  and the vertical maps  are induced by the canonical inclusion $\Gamma^* \subset T^*$.
\end{lem}

One has the well-known identification of the kernel and cokernel of $\diff$:

\begin{lem}
\label{lem:ker_coker_diff}
For positive integers $a,b$, the kernel and cokernel of $\diff$ identify respectively as 
\begin{eqnarray*}
\ker (\diff) &\cong & \schur_{(b+1,1^{a-1})}
\\
\mathrm{coker} (\diff) &\cong & \schur_{(b-1,1^{a+1})},
\end{eqnarray*} 
where, if $b=1$, the cokernel is understood to be zero.
\end{lem}

\begin{proof}
We use the Pieri formula (see Example \ref{exam:LR_Pieri}). This implies that, for positive integers $a$ and $b$, there is an isomorphism
$$
\Lambda^a \otimes \Gamma ^b 
\cong 
\schur _{(b+1,1^{a-1})}
\oplus 
\schur_{(b, 1^a)}.
$$ 
The differential $\diff : \Lambda^a \otimes \Gamma ^b \rightarrow \Lambda^{a+1} \otimes \Gamma^{b-1}$ is clearly non-zero and, by the above, identifies as a natural transformation
$$
\schur _{(b+1,1^{a-1})}
\oplus 
\schur_{(b, 1^a)}
\stackrel{\diff}{\rightarrow} 
\schur _{(b,1^{a})}
\oplus 
\schur_{(b-1, 1^{a+1})}.
$$
The result thus follows from Schur's lemma (Proposition \ref{prop:nat_schur} here).
\end{proof}

\begin{exam}
\label{exam:schur_embedding}
For $r \geq 2$,  $\schur_{(r, 1^2)}$ identifies as the kernel of 
$
\diff \colon  
\Lambda^3  \otimes \Gamma^{r-1} 
\rightarrow 
\Lambda^{4}  \otimes \Gamma^{r-2}.  
$

Hence, by Lemma \ref{lem:diff_characterization}, one has 
$$
\schur_{(r, 1^2)} \cong \ker \big( 
\Lambda^3  \otimes \Gamma^{r-1} 
\hookrightarrow 
\Lambda^3 \otimes T^{r-1} 
\cong \Lambda^3 \otimes \Lambda^1 \otimes T^{r-2}
\rightarrow 
\Lambda^4 \otimes T^{r-2}
\big), 
$$
where the first map is given by the inclusion $\Gamma^{r-1} \subset T^{r-1}$ and the second by the product $\Lambda^3 \otimes \Lambda^1 \rightarrow \Lambda^4$, as in that Lemma.
\end{exam}

\section{The adjoint action and a reduction}
\label{sect:adj}

The main purpose of this section is to reduce the calculation of 
 $H_1 (\lie (V); \lie(V)^{\otimes r})^{[r+2]}$ (for $r\geq 2$)  to that of the kernel of the morphism $\dbar_r$ introduced in Notation \ref{nota:dbar}; this is achieved in Proposition \ref{prop:H_1_by_dbar}.

\subsection{The adjoint action}

Consider $\lie (V)$, the free Lie algebra on $V\in \ob \fmodq$. Then $\lie (V)$ can be considered as a left $\lie (V)$-module for the adjoint action, and this structure is natural with respect to $V$. Hence, for each $r \in \nat$, $\lie (V)^{\otimes r}$ is a left $\lie (V)$-module (using the tensor product of $\lie (V)$-modules). For $r=0$ this gives the trivial module $\rat$ and, for  $r=1$,  the  adjoint representation. The symmetric group $\sym_r$ acts by place permutations of the tensor factors of $\lie (V)^{\otimes r}$, respecting the $\lie (V)$-module structure. All this structure is natural with respect to $V$. 

It follows that one can consider the Lie algebra homology 
$ 
H_* (\lie(V); \lie (V) ^{\otimes r}) 
$ 
and this takes values in $\sym_r$-modules,  naturally with respect to $V$.

Now, for any $\lie (V)$-module $M$, the Lie algebra homology $H_* (\lie (V); M)$ is  the homology of the complex $V \otimes M \rightarrow M$  (placed in homological degrees $1$ and $0$)  given by the restriction of the module structure. (This follows for example from the fact that $H_* (\lie (V); M)$ is naturally isomorphic to $\mathrm{Tor}_* ^{U \lie (V )} (\rat, M)$ and the universal enveloping algebra $U \lie (V)$ is naturally isomorphic to the tensor algebra $T(V)$. One concludes by using the projective resolution of $\rat$ as right $T(V)$-modules provided by $V \otimes T(V) \rightarrow T(V)$, where the map is given by  multiplication.) 

It follows that $ 
H_* (\lie(V); \lie (V) ^{\otimes r}) 
$  is  the homology of the complex
\[
V \otimes \lie (V) ^{\otimes r} 
\stackrel{\adbar_r}{\longrightarrow} 
\lie (V)^{\otimes r}
\]
given by restricting the adjoint action 
 $\mathrm{ad}_r : \lie (V) \otimes \lie (V) ^{\otimes r} 
\rightarrow 
\lie (V)^{\otimes r}$  to $V \subset \lie(V)$ on the first tensor factor. Thus one has 
\[
H_* (\lie (V); \lie(V)^{\otimes r}) \cong 
\left\{ 
\begin{array}{ll}
\coker (\adbar_r) & *=0 \\
\ker (\adbar _r) & *=1 \\
0 & \mbox{otherwise}.
\end{array}
\right.
\]

Since $\adbar_r$ is a natural transformation between  analytic functors (with respect to $V$), both $\ker (\adbar_r)$ and $\coker (\adbar_r)$ are analytic functors.

\begin{exam}
\label{exam:cyclie}
In the case $r=0$, the complex reduces to $
V \stackrel{0}{\rightarrow} \rat,  
$ 
giving the homology $H_* (\lie(V); \rat)$ with trivial coefficients. 

The case $r=1$ is more interesting; $\adbar_1$ fits into an exact complex:
\begin{eqnarray}
\label{eqn:exact_seq_cyclie}
0
\rightarrow 
\cyclie (V) 
\rightarrow 
V \otimes \lie (V)
\stackrel{\adbar_1}{\rightarrow} 
\lie (V) 
\rightarrow 
V 
\rightarrow 
0.
\end{eqnarray}
Here, $\cyclie(V)$ is the Schur functor that is associated to the cyclic Lie operad $\mathrm{CycLie}$. Recall that the cyclic Lie operad is constructed as for the Lie operad, but without a privileged `root' univalent vertex; namely $\mathrm{CycLie} (n)$ is given by the $\rat$-linear span of vertex-oriented, uni-trivalent  connected graphs, with univalent vertices labelled bijectively by $\{1, \ldots,  n\}$, modulo the AS (antisymmetry) and IHX (analogue of Jacobi) relations.

To explain this exact sequence, recall that, if $\liemod(n)$ denotes the $n$th Lie representation (corresponding to the functor $V \mapsto \lie(V)^{[n]}$ by the Schur correspondence, just as  $\mathrm{CycLie}(n)$ corresponds to $V \mapsto \cyclie(V)^{[n]}$), then one has the identification 
\[
\liemod(n) = \mathrm{CycLie}(n+1)\downarrow^{\sym_{n+1}}_{\sym_n},
\]
which shows how the Lie operad is obtained by designating a `root' univalent vertex. The collection of morphisms $\mathrm{CycLie}(n+1) \rightarrow \liemod(n)\uparrow^{\sym_{n+1}}_{\sym_n}$ obtained by adjunction (using that induction is equivalent to coinduction for finite groups) induce the natural transformation $\cyclie (V) \rightarrow V \otimes \lie (V)$  above and it is straightforward to check that one obtains a sequence as given.

 Exactness can be checked at the level of  representations of the symmetric groups by using restriction, exploiting the fact that there is an isomorphism of $\sym_n$-modules $\liemod(n+1) \downarrow ^{\sym_{n+1}}_{\sym_n} \cong \rat \sym_n$. Namely, if $n \geq 1$, the homogeneous polynomial degree $n+1$ part of (\ref{eqn:exact_seq_cyclie}) corresponds to the sequence of $\rat \sym_{n+1}$-modules:
\begin{eqnarray}
\label{eqn:cyclie_n+1_sequence}
\mathrm{CycLie}(n+1) 
\rightarrow 
\liemod (n)\uparrow _{\sym_n}^{\sym_{n+1}}
\rightarrow 
\liemod (n+1)
,
\end{eqnarray}
(the contribution from $V$ is zero). The  arrows are respectively injective and surjective, so we seek only to establish exactness in the middle. For this it is sufficient to show exactness after restricting to $\rat \sym_n$.

By the Pieri rule (or the Mackey formula), $\big(\liemod (n)\uparrow _{\sym_n}^{\sym_{n+1}}\big)\downarrow_{\sym_n}^{\sym_{n+1}}$ identifies as $\liemod(n) \oplus (\liemod(n)\downarrow_{\sym_{n-1}}^{\sym_{n}}\big)\uparrow _{\sym_{n-1}}^{\sym_{n}}$, hence as $\liemod(n) \oplus \rat \sym_{n}$. Thus, applying the restriction functor $\downarrow ^{\sym_{n+1}}_{\sym_n}$ to the sequence (\ref{eqn:cyclie_n+1_sequence}) gives the sequence of $\rat \sym_n$-modules
$$
\liemod (n) 
\rightarrow 
\liemod(n) \oplus \rat \sym_{n}
\rightarrow 
\rat \sym_n
.
$$
This is visibly short exact, whence the claimed exactness.  (The reader may prefer to rephrase the above simply as a dimension argument.) The exactness in the only non-trivial remaining case, corresponding to homogeneous polynomial degree one, is obvious.
\end{exam}

\begin{rem}
Stronger results (working over $\mathbb{Z}$) than those of Example \ref{exam:cyclie}  are established in   \cite{MR1943338,MR2240921}. Namely, in \cite{MR2240921} the author considers $\mathsf{D} (H)$, the kernel of the bracketing map $H \otimes \mathsf{L} (H) \rightarrow \mathsf{L} (H)$, where $\mathsf{L} (H)$ is the free Lie ring on the finitely-generated free abelian group $H$. He also considers $\mathcal{A}^t(H)$, defined similarly to  $\cyclie (V) $, but working over $\mathbb{Z}$. There is a comparison map $\mathcal{A}^t (H) \rightarrow \mathsf{D} (H)$ and the author gives precise information on the failure of this to be an isomorphism. In particular,  as a consequence of \cite[Theorem 2.5]{MR2240921}, he recovers the fact that $\mathcal{A}^t (H) \otimes \rat \rightarrow \mathsf{D} (H) \otimes \rat$ is an isomorphism. This corresponds to the exactness of (\ref{eqn:exact_seq_cyclie}).
\end{rem}

Here we are interested in $H_1 (\lie (V); \lie(V)^{\otimes r})^{[n]}$ for small $n$, considered as a functor of $V \in \ob \fmodq$ equipped with a $\sym_r$-action. Using the notation introduced in Section \ref{sect:prelim}, one has the basic result:

\begin{prop}
\label{prop:small_n}
For $r, n \in \nat$, there is a $\sym_r$-equivariant natural isomorphism:
\[
H_1 (\lie (V); \lie(V)^{\otimes r})^{[n]} = \left\{ 
\begin{array}{ll}
0 & n \leq r \\
\Gamma^{r+1} (V) \boxtimes \triv_r & n = r+1.
\end{array}
\right.
\]
\end{prop}

\begin{proof}
The case $n \leq r$ is clear, since $(V \otimes \lie(V) ^{\otimes r})^{[n]}$ is zero for $n \leq r$. 
 The case $n=r+1$ is shown as follows. The case $r=0$ is immediate.  For $r\geq 1$, one uses the
 fact that (for $r \geq 1$) $\adbar_r^{[r+1]}$ identifies as an explicit map
\begin{eqnarray}
\label{eqn:adbar_1}
V \otimes V^{\otimes r} \rightarrow \bigoplus_{i=0}^{r-1} V^{\otimes i} \otimes \Lambda^2 (V) \otimes  V^{\otimes r-1-i}.
\end{eqnarray}
The canonical inclusion $\Gamma^{r+1}(V) \hookrightarrow V \otimes V^{\otimes r}$ clearly maps to the kernel of this; moreover, the image is invariant under the action of $\sym_r$, by definition of the divided power functor. Hence, to prove the result, it suffices to show that, neglecting the $\sym_r$-action, the kernel is $\Gamma^{r+1}(V) $.

For $r=1$, the map (\ref{eqn:adbar_1}) gives the natural short exact sequence
\[
0
\rightarrow \Gamma^2 (V) \rightarrow V^{\otimes 2} \stackrel{\adbar_1}{\rightarrow}
\Lambda^2 (V) 
\rightarrow 
0,
\]
identifying the kernel as required. The cases $r>1$ then follow by an inductive argument (cf. the proof of Lemma \ref{lem:intersect_lambda}).
\end{proof}

The  case $n=r+2$ stated below is the main result of this paper (Theorem \ref{THM} of the Introduction). 

\begin{thm}
\label{thm:main}
For $r \in \nat$, there is a $\sym_r$-equivariant natural isomorphism:
\[
H_1 (\lie (V); \lie(V)^{\otimes r})^{[r+2]}
\cong 
\left\{ 
\begin{array}{ll}
0 & r=0\\
\Lambda^3 (V) \boxtimes \sgn_1 & r=1 \\
\Lambda^{r+2} (V) \boxtimes \sgn_r \   \oplus \  \schur_{(r, 1^2)} (V) \boxtimes \triv_r
& r>1.
\end{array}
\right.
\]
\end{thm}

\begin{exam}
\label{exam:r=1}
The case $r=1$ of Theorem \ref{thm:main} is elementary. The requisite complex is 
\[
V \otimes \Lambda^2 (V) \stackrel{\adbar_1^{[3]}}{\twoheadrightarrow} \schur_{(2,1)}(V),
\]
using the identifications $\lie(V)^{[2]} = \Lambda^2 (V)$ and $\lie(V)^{[3]} = \schur_{(2,1)}(V)$; the map is surjective since $\lie (V)$ is generated as a Lie algebra by $V$.
 By the Pieri rule, $V \otimes \Lambda^2 (V) \cong \Lambda^3(V) \oplus  \schur_{(2,1)}(V)$, giving the identification of the kernel.
\end{exam}

\subsection{Reduction of the problem}
\label{subsect:reduce}

To calculate $H_1 (\lie (V); \lie(V)^{\otimes r})^{[r+2]}$, we need to calculate the kernel of 
\[ 
\adbar_r ^{[r+2]} \ : \ 
\Big( V \otimes \lie (V)^{\otimes r} \Big)^{[r+2]} 
\rightarrow 
\Big (\lie(V)^{\otimes r} \Big) ^{[r+2]}.
\]

This can be described explicitly as follows:

\begin{lem}
\label{lem:identify_[r+2]} 
For $r \geq 1$, there are natural $\sym_r$-equivariant isomorphisms
\begin{eqnarray*}
\Big( V \otimes \lie (V)^{\otimes r} \Big)^{[r+2]} 
&\cong & 
V \otimes (\Lambda^2 (V) \otimes V^{\otimes r-1} ) \uparrow_{\sym_{r-1}}^{\sym_r}
\\
\Big (\lie(V)^{\otimes r} \Big) ^{[r+2]}
&\cong& 
 (\schur_{(2,1)} (V) \otimes V^{\otimes r-1} ) \uparrow_{\sym_{r-1}}^{\sym_r}
\ \quad \oplus \ \quad (\Lambda^2 (V) \otimes \Lambda^2 (V) \otimes V^{\otimes r-2}) \uparrow_{\sym_2 \times \sym_{r-2}}^{\sym_r},
 \end{eqnarray*}
 for the obvious actions of the symmetric groups by place permutations (if $r =1$, the final term is understood to be zero).
\end{lem}

\begin{proof}
We use the identifications in Example \ref{exam:lie_small_degree}, in particular $\lie (V)^{[0]} = 0$ and $\lie (V)^{[1]} = V$.

For the first point, we use the isomorphism $\big(V \otimes \lie (V)^{\otimes r} \big)^{[r+2]} \cong  V \otimes (\lie(V)^{\otimes r})^{[r+1]}$. 
Then we have the identification
$$
(\lie(V)^{\otimes r})^{[r+1]}
\ \cong \ 
\bigoplus_{\substack{(a_1, \ldots ,a_r)\in \nat_*^r \\ \sum_i a_i =r+1}} \bigotimes_{i=1}^r \lie (V)^{[a_i]},
$$  
where the sum is over sequences of $r$ positive integers that sum to $r+1$. The only solutions to the latter are when exactly one $a_i$ is $2$, all the others being equal to $1$. Since $\lie(V)^{[2]} \cong \Lambda^2 (V)$, taking into account the $\sym_r$-action, this gives the stated result. 

 The second identification is proved similarly; in this case, the sum is over sequences $(0<a_i \mid 1 \leq i \leq r)$ such that $\sum_{i=1}^r a_i = r+2$. There are two different types of solution: either exactly one of the $a_i$'s is equal to $3$ and the others are all equal to $1$; or exactly two of the $a_i$'s are equal to $2$ and the others are all equal to $1$. Since $\lie (V)^{[3]} \cong \schur_{(2,1)} (V)$, the result follows as above.
\end{proof}

The  adjoint action $\adbar_1$ gives the canonical projections
\begin{eqnarray*}
(V\otimes \lie (V))^{[2]} \cong V \otimes V &\twoheadrightarrow &\lie (V)^{[2]} = \Lambda^2 (V) \\ 
(V \otimes \lie (V))^{[3]} \cong V \otimes \Lambda^2 (V) &\twoheadrightarrow &\lie (V)^{[3]} = \schur_{(2,1)} (V), 
\end{eqnarray*} 
with kernels $\Gamma^2 (V)$ and $\Lambda^3 (V)$ respectively. These extend to define the components $\delta'_r$ and $\delta''_r$ of $\adbar_r^{[r+2]}$:
\begin{eqnarray*}
\delta'_r &:&  V \otimes (\Lambda^2  (V) \otimes V^{\otimes r-1})\uparrow _{\sym_{r-1} }^{\sym_r} 
\rightarrow 
(\schur_{(2,1)}  (V) \otimes V^{\otimes r-1})\uparrow _{\sym_{r-1} }^{\sym_r} 
\\
\delta''_r &:& V \otimes (\Lambda^2  (V) \otimes V^{\otimes r-1})\uparrow _{\sym_{r-1} }^{\sym_r} 
\rightarrow (\Lambda^2  (V) \otimes \Lambda^2 (V) \otimes  V^{\otimes r-2})\uparrow _{\sym_2 \times \sym_{r-2} }^{\sym_r}.
\end{eqnarray*}

By Frobenius reciprocity (i.e., using that $\uparrow _{\sym_{r-1} }^{\sym_r} $ is left adjoint to $\downarrow _{\sym_{r-1} }^{\sym_r}$), these are equivalent to $\sym_{r-1}$-equivariant maps: 
 \begin{eqnarray*}
\widetilde{\delta'_r} &:&  V \otimes (\Lambda^2  (V) \otimes V^{\otimes r-1})
\rightarrow 
(\schur_{(2,1)}  (V) \otimes V^{\otimes r-1})\uparrow _{\sym_{r-1} }^{\sym_r} 
\\
\widetilde{\delta''_r} &:& V \otimes (\Lambda^2  (V) \otimes V^{\otimes r-1})
\rightarrow (\Lambda^2  (V) \otimes \Lambda^2 (V) \otimes  V^{\otimes r-2})\uparrow _{\sym_2 \times \sym_{r-2} }^{\sym_r},
\end{eqnarray*}
using the appropriate restricted $\sym_{r-1}$-module structure on the codomains (not indicated in the notation).

The morphism $\widetilde{\delta'_r}$ is easily described: it is the composite of 
$$
V \otimes \Lambda^2  (V) \otimes V^{\otimes r-1}
\rightarrow 
\schur_{(2,1)}  (V) \otimes V^{\otimes r-1}$$ 
(given by applying $\adbar_1$ to the first two tensor factors of the domain) with the obvious $\sym_{r-1}$-equivariant inclusion $ \schur_{(2,1)}  (V) \otimes V^{\otimes r-1}
\subset (\schur_{(2,1)}  (V) \otimes V^{\otimes r-1})\uparrow _{\sym_{r-1} }^{\sym_r}$.

Likewise, the morphism  $\widetilde{\delta''_r}$ identifies as the composite of 
\[
V \otimes ( \Lambda^2  (V) \otimes V^{\otimes r-1})
\rightarrow 
\bigoplus_{1 < k \leq r} 
(\Lambda^2 (V) \otimes V^{\otimes k-2} \otimes \Lambda^2 (V) \otimes V^{\otimes r-k})
\]
(the $k$th component of which is given by applying $\adbar_1$ to the first and $(k+1)$st tensor factors of the domain) with the $\sym_{r-1}$-equivariant inclusion of the codomain into $(\Lambda^2  (V) \otimes \Lambda^2 (V) \otimes  V^{\otimes r-2})\uparrow _{\sym_2 \times \sym_{r-2} }^{\sym_r}$.

\begin{lem}
\label{lem:ker_delta'}
The kernel of $\delta'_r$ is  $\sym_r$-equivariantly naturally isomorphic to
$
(\Lambda^3  (V) \otimes V^{\otimes r-1})\uparrow _{\sym_{r-1} }^{\sym_r}, 
$
which is considered as a $\sym_r$-submodule of $V \otimes (\Lambda^2 (V) \otimes V^{\otimes r-1})\uparrow _{\sym_{r-1} }^{\sym_r}$ via the inclusion induced up from the $\sym_{r-1}$-equivariant map  
\[
\Lambda^3  (V) \otimes V^{\otimes r-1} \hookrightarrow V \otimes \Lambda^2 V \otimes V^{\otimes r-1}
\] 
given by applying the inclusion $\Lambda^3 (V) 
\rightarrow 
V \otimes \Lambda^2 (V) $ to the first factor.
\end{lem}

\begin{proof}
This generalizes the case $r=1$ given in Example \ref{exam:r=1}. 
\end{proof}

It is convenient to write the expressions without the variable $V$; for example, $V^{\otimes r-1}$ is replaced by the functor $T^{r-1}$.

\begin{nota}
\label{nota:dbar}
Write $\dbar_r$ for the restriction of $\delta''_r$ to the kernel of $\delta'_r$:
\[
\dbar_r \ : \ 
(\Lambda^3   \otimes T^{ r-1})\uparrow _{\sym_{r-1} }^{\sym_r}
 \rightarrow 
(\Lambda^2   \otimes \Lambda^2  \otimes  T^{ r-2})\uparrow _{\sym_2 \times \sym_{r-2} }^{\sym_r}.
\] 
\end{nota}

The underlying functors of the domain and codomain of $\dbar_r$ identify explicitly as follows:

\begin{lem}
\label{lem:domain_codomain_dbar}
There are isomorphisms of functors 
\begin{eqnarray*}
(\Lambda^3  \otimes T^{r-1})\uparrow _{\sym_{r-1} }^{\sym_r}
&\cong &
\bigoplus_{i=1}^r 
T^{i-1} \otimes \Lambda^3 \otimes T^{r-i}
\\
(\Lambda^2   \otimes \Lambda^2 \otimes  T ^{ r-2})\uparrow _{\sym_2 \times \sym_{r-2} }^{\sym_r}
& \cong &
\bigoplus 
_{1 \leq i< j \leq r} 
T^{i-1} \otimes \Lambda^2 \otimes T^{j-i-1} \otimes \Lambda^2 \otimes T^{r-j}.
\end{eqnarray*}
Thus, $\dbar_r$ identifies as a natural transformation
\[
\bigoplus_{i=1}^r 
T^{i-1} \otimes \Lambda^3 \otimes T^{r-i}
\rightarrow 
\bigoplus 
_{1 \leq i< j \leq r} 
T^{i-1} \otimes \Lambda^2 \otimes T^{j-i-1} \otimes \Lambda^2 \otimes T^{r-j}.
\]
\end{lem}

Putting the above together provides the following reduction to the study of $\ker \dbar_r$.

\begin{prop}
\label{prop:H_1_by_dbar}
For $r \geq 2$, $H_1 (\lie (V); \lie(V)^{\otimes r})^{[r+2]}$ is naturally isomorphic to $\ker \dbar_r (V)$.
\end{prop}

\subsection{A geometric visualization}
\label{subsect:geometric_vis}

Henceforth, we use without  further comment the identifications $\Lambda^1 (V) = V = T^1 (V)$, so that, for $\ell \in \nat$,  $T^\ell$ is naturally isomorphic to $(\Lambda^1)^{\otimes \ell}$, understood to be $\rat$ for $\ell=0$.

It is useful to visualize the maps in $\dbar_r$ using the $1$-skeleton of an $r-1$-simplex, considered as a directed bipartite graph as follows. The term $T^{i-1} \otimes \Lambda^3 \otimes T^{r-i}$ labels the black vertex of the $(r-1)$-simplex  labelled by $i \in \{ 1, \ldots r\}$ (using this indexing set for compatibility with that used elsewhere). For each edge of the $1$-skeleton, one adds a white vertex; on the edge between the black vertices  $i< j$ this is labelled by the term $T^{i-1}\otimes \Lambda^2 \otimes T^{ j-i-1} \otimes \Lambda^2 \otimes T^{ r-j}$. This white vertex receives directed edges as indicated in the example below for $r=3$:

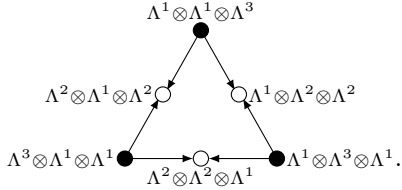
\begin{tikzpicture}[scale = 1]
 \draw [-latex] (0,0) -- (.9,0);
 \draw [latex-] (1.1,0) -- (2,0);
  \draw [-latex](0,0)--(.45,.8);
  \draw [latex-] (.55,.9) -- (1, 1.7);
  \draw [-latex](1, 1.7)--(1.45,.9);
  \draw [latex-] (1.55,.8) -- (2,0);
    \draw [fill=white] (1.5,.85) circle [radius = .1] node [right]{$\scriptstyle{\Lambda^1\otimes\Lambda^2  \otimes  \Lambda^2}$} ; 
    \draw [fill=black] (0,0) circle [radius = .1] node [left] {$\scriptstyle{\Lambda^3 \otimes \Lambda^1 \otimes \Lambda^1}$};
  \draw [fill=white] (1,0) circle [radius = .1] node [below] {$\scriptstyle{\Lambda^2 \otimes \Lambda^2 \otimes \Lambda^1}$};  
  \draw [fill=black] (2,0) circle [radius = .1] node [right] {$\scriptstyle{\Lambda^1 \otimes \Lambda^3 \otimes \Lambda^1}$.};
  \draw [fill=black] (1,1.7) circle [radius = .1] node [above] {$\scriptstyle{\Lambda^1  \otimes \Lambda^1\otimes \Lambda^3}$};
   \draw [fill=white] (.5,.85) circle [radius = .1] node [left] {$\scriptstyle{\Lambda^2  \otimes \Lambda^1\otimes \Lambda^2}$};
\end{tikzpicture}

This allows visualization of  `restriction' corresponding to faces. For instance, if one forgets all vertices that have either $\Lambda^2 $ or $\Lambda^3 $ in the last factor, one is left with:

\begin{tikzpicture}[scale = 1]
 \draw [-latex] (0,0) -- (.9,0);
 \draw [latex-] (1.1,0) -- (2,0);
    \draw [fill=black] (0,0) circle [radius = .1] node [left] {$\scriptstyle{\Lambda^3 \otimes \Lambda^1 \otimes \Lambda^1}$};
  \draw [fill=white] (1,0) circle [radius = .1] node [below] {$\scriptstyle{\Lambda^2 \otimes \Lambda^2 \otimes \Lambda^1}$};  
  \draw [fill=black] (2,0) circle [radius = .1] node [right] {$\scriptstyle{\Lambda^1 \otimes \Lambda^3 \otimes \Lambda^1}$.};
 \end{tikzpicture}
 
 \noindent
This is obtained from the corresponding diagram for  $r=2$ by applying $-\otimes \Lambda^1$.

\section{Analysing $\dbar_r$}
\label{sect:kappa}

This section first explains how to analyse the kernel of $\dbar_r$ (see Proposition \ref{prop:X_i_ker_dbar}). This is based on the morphism $\alpha : \Lambda^3 \otimes \Lambda^1 \rightarrow \Lambda^1 \otimes \Lambda^3$ that is introduced in Section \ref{subsect:alpha} (for the case $r=2$), as well as its analogues $\alpha_{i;j}$ (for $r \geq 2$ and $1 \leq i<j \leq r$) introduced in Section \ref{subsect:extend_alpha}.

This is applied in Theorem \ref{thm:r=2} to establish the case $r=2$ of Theorem \ref{THM}. It is also used in Proposition \ref{prop:lower_bound_ker} to establish a lower bound for $\ker \dbar_r$ that will subsequently be shown to be an equality.

\subsection{The functors $\Lambda^3 \otimes \Lambda^1$ and $\Lambda^2 \otimes \Lambda^2$}

We start by reviewing the decomposition of $\Lambda^3 \otimes \Lambda^1$ and $\Lambda^2 \otimes \Lambda^2$ into simple functors, paying attention to the $\sym_2$-action in the second case.

\begin{lem}
\label{lem:decompose}
There is a canonical splitting
$
\Lambda^3 \otimes \Lambda^1 \cong  \Lambda^4  \oplus \schur_{(2,1^2)} 
$ 
and  a canonical  $\sym_2$-equivariant splitting
\[
\Lambda^2 \otimes \Lambda^2 
\cong 
\big( \Lambda^4   \oplus \schur_{(2,2)}\big) \boxtimes \triv_2 
\ 
\oplus 
\ 
\schur_{(2,1^2)}  \boxtimes \sgn_2,
\]
where $\sym_2$ acts by place permutations on the left hand side.
\end{lem}

\begin{proof}
The splitting $ \Lambda^3 \otimes \Lambda^1  \cong \Lambda^4  \oplus  \schur_{(2,1^2)}$ and the underlying splitting $\Lambda^2 \otimes \Lambda^2 \cong \Lambda^4   \oplus \schur_{(2,2)} \oplus \schur_{(2,1^2)} $ follow from Pieri's rule; they are canonical since the functors are multiplicity-free. It remains to consider how the $\sym_2$-action on $\Lambda^2  \otimes \Lambda^2 $ behaves with respect to the latter. By Schur's lemma (Proposition \ref{prop:nat_schur} for Schur functors), this action corresponds to an invertible element of 
\begin{eqnarray}
\label{eqn:End_1^4_2,2}
\mathrm{End} (\Lambda^4   \oplus \schur_{(2,2)} \oplus \schur_{(2,1^2)}) 
\cong 
\rat_{(1^4)} \prod \rat_{(2,2)} \prod \rat_{(2,1^2)},
\end{eqnarray}
where the factors of the product of copies of $\rat$ are indexed by the partitions to which they correspond. More precisely, the element is of the form $(u , v, w) \in \rat\prod \rat \prod \rat$, where $u, v, w \in \{1, -1 \}$.

There is an inclusion of $\Lambda^4$ into $\Lambda^2\otimes \Lambda^2$ given by the coproduct. The image is clearly invariant under transposition of tensor factors hence induces $\Lambda^4 \hookrightarrow  (\Lambda^2 \otimes \Lambda^2)^{\sym_2}$.  

Since $(\Lambda^2 \otimes \Lambda^2 )^{\sym_2}$ evaluated on  $\rat^{\oplus 2}$ is clearly non-zero, it must also contain $\schur_{(2,2)}$ as a composition factor, since $\schur_{(2,1^2)}$ and $\Lambda^4$ are both zero when evaluated on $\rat^{\oplus 2}$, by Lemma \ref{lem:conn_schur}.

To conclude, it suffices to observe that $\Lambda^2  \otimes \Lambda^2$ is not invariant under the $\sym_2$-action transposing factors. This implies that the remaining factor must correspond to $\schur_{(2,1^2)}  \boxtimes \sgn_2$. 
\end{proof}

The direct sum decomposition of $\Lambda^3 \otimes \Lambda^1$ can be made explicit as follows. 

\begin{lem}
\label{lem:idempotent_e}
\ 
\begin{enumerate}
\item 
The idempotent corresponding to the projection of $\Lambda^3 \otimes \Lambda^1  $ onto the factor $\Lambda^4  $ is 
\[
e:= 
\frac{1}{4} \Big( 
\Lambda^3 \otimes \Lambda^1 
\stackrel{\mu} {\rightarrow} 
\Lambda^4 
\stackrel{\psi}{\rightarrow}
\Lambda^3  \otimes \Lambda^1 
\Big),
\]
where $\mu$ is the product and $\psi$ the coproduct.
 Explicitly, $e (x\wedge y \wedge z \otimes w) = \frac{1}{4} \big( x\wedge y \wedge z \otimes w
- x \wedge y \wedge w \otimes z + x \wedge z \wedge w \otimes y - y \wedge z \wedge w \otimes x\big).$
\item 
The idempotent giving the projection onto $\schur_{(2,1^2)} $ is $(\id- e)$. This is given explicitly by $(\id- e) (x\wedge y \wedge z \otimes w) = \frac{1}{4} \big( 3  x\wedge y \wedge z \otimes w
+ x \wedge y \wedge w \otimes z - x \wedge z \wedge w \otimes y  + y \wedge z \wedge w \otimes x\big)$. 
\end{enumerate}
\end{lem}

\begin{proof}
It is a  straightforward verification that $e$ is an idempotent; since it factors across $\Lambda^4$ it gives the required projection. The second statement follows immediately.
\end{proof}

We record the following application of Schur's lemma to endomorphisms of $\Lambda^3 \otimes \Lambda^1$ analogous to (\ref{eqn:End_1^4_2,2}) above:

\begin{lem}
\label{lem:schur_Lambda3_otimes_Lambda1}
The endomorphism ring of $\Lambda^3 \otimes \Lambda^1  \cong \Lambda^4 \oplus \schur_{(2,1^2)}$ is canonically isomorphic to the product of two copies of $\rat$:
\begin{eqnarray}
\label{eqn:End_3otimes1}
\mathrm{End} (\Lambda^3 \otimes \Lambda^1) \cong 
\mathrm{End}(\Lambda^4) \prod \mathrm{End} (\schur_{(2,1^2)}) = \rat_{(1^4)} \prod \rat_{(2,1^2)}.
\end{eqnarray}
\end{lem}

\subsection{The morphisms $\kappa$ and $\tilde{\kappa}$}

We introduce the morphisms $\kappa$ and  $\tilde{\kappa}$ that correspond to the components of $\dbar_2$:

\begin{nota}
\label{nota:kappa}
Let $\kappa : \Lambda^3  \otimes \Lambda^1  \rightarrow \Lambda ^2  \otimes \Lambda^2  $ denote the composite of the coproduct on the first factor followed by the product on the last factors:
\begin{eqnarray}
\Lambda^3  \otimes \Lambda^1 
\stackrel{\psi \otimes \id}{\rightarrow} 
\Lambda^2  \otimes \Lambda^1   \otimes \Lambda^1 
\stackrel{\id \otimes \mu}{\rightarrow}
 \Lambda ^2  \otimes \Lambda^2 .
\end{eqnarray}

The morphism $\tilde{\kappa} : \Lambda^1 \otimes \Lambda^3  \rightarrow \Lambda^2 \otimes \Lambda^2 $ is defined by the commutative diagram:
\[
\xymatrix{
\Lambda^1 \otimes \Lambda^3 \ar[r]^{\tilde{\kappa}}
\ar[d]_\tau
&
\Lambda^2  \otimes \Lambda^2 
\ar[d]^\tau 
\\
\Lambda^3  \otimes \Lambda^1  \ar[r]^{\kappa}
&
\Lambda^2  \otimes \Lambda^2  ,
}
\]
in which the maps labelled $\tau$ exchange the tensor factors.
\end{nota}

The significance of $\kappa $ is shown by the following Lemma, which is an addendum to Lemma \ref{lem:domain_codomain_dbar}:

\begin{lem}
\label{lem:dbar}
For $r\geq 2$, the natural transformation $\dbar_r$ is the unique $\sym_r$-equivariant natural transformation with component 
\[
\Lambda^3  \otimes T^{r-1} 
\rightarrow 
\Lambda^2 \otimes \Lambda^2 \otimes T^{ r-2}
\]
given by $\kappa \otimes \id_{T^{r-2}}$.
\end{lem}

\begin{rem}
There is a slight subtlety here: $\kappa$ is defined using the coproduct $\Lambda^3 \rightarrow \Lambda^2 \otimes \Lambda^1$, whereas $\dbar_2$ is defined using the coproduct $\Lambda^3 \rightarrow \Lambda^1 \otimes \Lambda^2$. This makes no difference, since the `graded commutativity' of the exterior algebra in this case ensures that the two maps identify via the transposition of the tensor factors of $\Lambda^1 \otimes \Lambda^2$.
\end{rem}

The definition of $\tilde\kappa$ is justified by the following:

\begin{lem}
\label{lem:r=2}
The natural transformation underlying 
$\dbar_2$ identifies as 
\[
(\kappa,\tilde{\kappa}) 
\ : \ 
\Lambda^3  \otimes \Lambda^1 
\ \oplus \ 
\Lambda^1  \otimes \Lambda^3 
\rightarrow 
\Lambda^2 \otimes \Lambda^2  .
\]
\end{lem}

\begin{proof}
This follows from the characterization of  $\dbar_2$ given by Lemma \ref{lem:dbar} by showing that the map $(\kappa,\tilde{\kappa}) $  is $\sym_2$-equivariant. This property follows from the construction of $\tilde{\kappa}$ in terms of $\kappa$.
\end{proof}

\subsection{The isomorphism $\alpha$}
\label{subsect:alpha}

Before introducing  the announced isomorphism $\alpha$, we require to establish some properties of $\kappa$ and $\tilde{\kappa}$:

\begin{lem}
\label{lem:injective}
The natural transformations 
\begin{eqnarray*}
\kappa &:& \Lambda^3 \otimes \Lambda^1  \rightarrow \Lambda ^2  \otimes \Lambda^2 
\\
\tilde{\kappa} &:& \Lambda^1 \otimes \Lambda^3  \rightarrow \Lambda ^2  \otimes \Lambda^2 
\end{eqnarray*}
are injective. Moreover 
$
\mathrm{image} (\kappa) = \mathrm{image} (\tilde{\kappa}) 
$.
\end{lem}

\begin{proof}
To prove that $\kappa$ is injective, using Lemma \ref{lem:decompose}, it suffices to show that $\kappa|_{\Lambda^4}$ and $\kappa|_{\schur_{(2,1^2)}}$ are both non-trivial, since $\Lambda^4 $ and $\schur_{(2,1^2)}$ are non-isomorphic simple functors (see Proposition \ref{prop:nat_schur}). For $\Lambda^4$, it suffices to check that the composite $\Lambda^4  \stackrel{\psi}{\rightarrow } \Lambda^3  \otimes \Lambda^1 \stackrel{\kappa}{\rightarrow} \Lambda ^2  \otimes \Lambda^2 $ is non-trivial; this is a simple calculation.  For $\schur_{(2,1^2)}$, it suffices to check that $\kappa$ is non-trivial when evaluated on $V = \rat^{\oplus 3}$; this is again straightforward.

That $\tilde{\kappa}$ is injective follows immediately from the definition of $\tilde{\kappa}$ in terms of $\kappa$. By inspection, again using Lemma \ref{lem:decompose}, the images coincide. 
\end{proof}

This gives the commutative diagram:

\begin{eqnarray}
\label{eqn:alpha_diagram}
\xymatrix{
\Lambda^3 \otimes \Lambda^1 
\ar[r]^\kappa 
\ar[rd]_\cong 
&
\Lambda^2  \otimes \Lambda^2
&
\Lambda^1 \otimes \Lambda^3
\ar[l]_{\tilde{\kappa}}
\ar[ld]^\cong 
\\
&
\mathrm{image} (\kappa) =  \mathrm{image} (\tilde{\kappa}) .
\ar@{^(->}[u]
}
\end{eqnarray}

\begin{defn}
\label{defn:alpha}
Let $\alpha$ be the isomorphism $\Lambda^3  \otimes \Lambda^1  \stackrel{\cong}{\rightarrow} 
\Lambda^1  \otimes \Lambda^3$ induced by passage around the bottom of (\ref{eqn:alpha_diagram}) (inverting the right hand isomorphism).
\end{defn}

Thus $\tau \circ \alpha$ is an automorphism of $\Lambda^3 \otimes \Lambda^1$, using the symmetry $\Lambda^1 \otimes \Lambda^3 \stackrel{\tau}{\rightarrow} \Lambda^3 \otimes \Lambda^1$.
By Lemma \ref{lem:schur_Lambda3_otimes_Lambda1}, this identifies as follows:

\begin{prop}
\label{prop:new_tau_alpha}
Under the isomorphism (\ref{eqn:End_3otimes1}), the automorphism $\tau \circ \alpha \in \mathrm{End} (\Lambda^3 \otimes \Lambda^1)$ identifies as $(1, -1)$. In particular,  $\tau \circ \alpha$ is an involution, i.e.,  $(\tau \circ \alpha) \circ (\tau \circ \alpha) = \id$. 
\end{prop}

\begin{proof}
Using the definition of $\tilde{\kappa}$ in terms of $\kappa$, this follows from Lemma \ref{lem:decompose}, which describes the action of $\sym_2$ on $\Lambda^2 \otimes \Lambda^2$ by place permutation in terms of the isotypical decomposition of $\Lambda^2 \otimes \Lambda^2$.
\end{proof}

This allows the following explicit identification of the automorphism $\tau \circ \alpha$, using the idempotent $e \in \mathrm{End} (\Lambda^3\otimes \Lambda^1)$ of Lemma \ref{lem:idempotent_e}:

\begin{cor}
\label{cor:new_tau_alpha}
The automorphism $\tau \circ \alpha$ of $\Lambda^3  \otimes \Lambda^1$ identifies as the morphism $2 e - \id$. Explicitly:
 $$\tau \circ \alpha (x \wedge y \wedge z \otimes w ) = 
-\frac{1}{2} \big( x \wedge y \wedge z  \otimes w + x \wedge y \wedge w \otimes z - x \wedge z \wedge w \otimes y + y \wedge z \wedge w \otimes x \big).$$ 
\end{cor}

\subsection{Extending $\alpha$ for $r>3$}
\label{subsect:extend_alpha}

The isomorphism $\alpha$ was introduced for the case $r=2$, where one only has to consider the terms $\Lambda^3  \otimes \Lambda^1 $ and $\Lambda^1  \otimes \Lambda^3$. More generally, for each $1 \leq i \leq r$, one has
$
T^{ i-1} \otimes \Lambda^3 \otimes T^{ r-i}
$  
in which $\Lambda^3 $ appears as the $i$th tensor factor.

\begin{nota}
\label{nota:alpha_ij}
For integers $1 \leq i < j \leq r$, let $\alpha_{i;j}$ denote the isomorphism
\[
\alpha_{i;j} : T ^{ i-1} \otimes \Lambda^3  \otimes T ^ {r-i}
\rightarrow 
T^{ j-1} \otimes \Lambda^3 \otimes T^{r-j}
\]
given by applying $\alpha : \Lambda^3  \otimes \Lambda^1  \rightarrow \Lambda^1  \otimes \Lambda^3$ with respect to the $i$th and $j$th tensor factors and acting by the identity on the remaining factors.
\end{nota}

\begin{rem}
\ 
\begin{enumerate}
\item 
We only require to define $\alpha_{i;j}$ for $i<j$, since we can use $\alpha_{j;i}^{-1}$ to cover the cases $i>j$. 
\item 
By Proposition \ref{prop:new_tau_alpha},  we have $\alpha_{j;i}^{-1} = \tau_{j;i} \circ \alpha_{j;i} \circ \tau_{j:i}$, where $\tau_{j;i}$ transposes the $i$th and $j$th tensor factors.
\end{enumerate}
\end{rem}

\begin{exam}
\label{exam:alpha}
\ 
\begin{enumerate}
\item 
For $r=2$, one has $\alpha_{1;2} = \alpha$. 
\item 
If $r>2$, then $\alpha_{1:2}$ is $\alpha \otimes \id_{T^{ r-2}}$.
\end{enumerate}
\end{exam}

The significance of $\alpha$ is shown by the following 

\begin{lem}
\label{lem:alpha_r=2}
Consider $X_1 \in \Lambda^3 \otimes \Lambda^1$ and $X_2 \in \Lambda^1 \otimes \Lambda^3$. 
Then $(X_1 , X_2) \in \Lambda^3 \otimes \Lambda^1 \ \oplus \ \Lambda^1 \otimes \Lambda^3$ lies in the kernel of $\dbar_2$ if and only if 
$$
\alpha_{1;2} (X_1) = \alpha (X_1)  = - X_2.
$$
\end{lem}

\begin{proof}
By Lemma \ref{lem:r=2}, $(X_1 , X_2) \in \Lambda^3 \otimes \Lambda^1 \ \oplus \ \Lambda^1 \otimes \Lambda^3$ lies in the kernel of $\dbar_2$ if and only if $\kappa (X_1) + \tilde\kappa (X_2) =0$. By the definition of $\alpha$, this is equivalent to the condition in the statement, using the identification $\alpha_{1;2}= \alpha$ observed in Example \ref{exam:alpha}.
\end{proof}

\subsection{Identifying $\ker \dbar_r$}

Suppose that $r \geq 2$. Motivated by Proposition \ref{prop:H_1_by_dbar}, we seek to calculate $\ker \dbar_r$, where $\dbar_r$ is the morphism introduced in Notation \ref{nota:dbar}. The latter has the form given in Lemma \ref{lem:domain_codomain_dbar}:
\[
\dbar_r \colon 
\bigoplus_{i=1}^r 
T^{i-1} \otimes \Lambda^3 \otimes T^{r-i}
\rightarrow 
\bigoplus 
_{1 \leq i< j \leq r} 
T^{i-1} \otimes \Lambda^2 \otimes T^{j-i-1} \otimes \Lambda^2 \otimes T^{r-j}.
\]
This is described explicitly by Lemma \ref{lem:dbar} in terms of $\kappa$. 

Given a polynomial functor $F$,  we will use the shorthand `element' to refer to an element of $F(V)$ for some $V \in \ob \fmodq$, and write $x \in F$ for such an element.

\begin{nota}
\label{nota:X_i}
Write an element of $\bigoplus_{i=1}^r 
T^{i-1} \otimes \Lambda^3 \otimes T^{r-i}$ as $(X_i)_{1 \leq i \leq r}$, where $X_i \in T^{i-1} \otimes \Lambda^3 \otimes T^{r-i}$.
\end{nota}

\begin{prop}
\label{prop:X_i_ker_dbar}
Suppose that $r \geq 2$. An element $(X_i)_{1 \leq i \leq r}$ of $\bigoplus_{i=1}^r 
T^{i-1} \otimes \Lambda^3 \otimes T^{r-i}$ lies in $\ker \dbar_r$ if and only if, for all $1 \leq i < j \leq r$, one has 
\[
\alpha_{i;j}( X_i ) =  - X_j.
\]
\end{prop}

\begin{proof}
The case $r=2$ has already been established by Lemma \ref{lem:alpha_r=2}; the general case extends this, as explained below.

The situation can be visualized by using the $1$-skeleton of the $(r-1)$-simplex,  as in Section \ref{subsect:geometric_vis}. The elements $X_i$ label the black vertices (for $1 \leq i \leq r$) and the arrows correspond to the components of $\dbar_r$; a white vertex corresponds to a pair $i < j$.

By the previous identifications, the element $(X_i)_{1 \leq i \leq r}$ lies in the kernel of $\dbar_r$ if and only if, for each $i<j$, the respective images of $X_i$ and $X_j$ in the corresponding white vertex sum to zero. 

Explicitly, the pair $(X_i, X_j)$ must lie in the kernel of the corresponding map 
\[
(T^{i-1} \otimes \Lambda^3 \otimes T^{r-i})
\ 
\oplus 
\ 
(T^{j-1} \otimes \Lambda^3 \otimes T^{r-j})
\rightarrow 
T^{i-1} \otimes \Lambda^2 \otimes T^{j-i-1} \otimes \Lambda^2 \otimes T^{r-j}. 
\] 
The component maps are induced by $\kappa$ and $\tilde{\kappa}$ applied with respect to the $i$th and $j$th tensor factors. For example, in the case $i=1$, $j=2$, the components identify respectively as $\kappa \otimes \id_{T^{r-2}}$ and $\tilde{\kappa} \otimes \id_{T^{r-2}}$. 

As in the proof of the case $r=2$ in Lemma \ref{lem:alpha_r=2}, the previous condition is equivalent to $\alpha_{i;j} (X_i) = -X_j$, by the definition of $\alpha_{i;j}$.
\end{proof}

This has the immediate consequence:

\begin{cor}
\label{cor:ell_inclusion}
Suppose that $r \geq 2$. For any $1 \leq  \ell \leq r$, the projection 
\[
\bigoplus_{i=1}^r 
T^{i-1} \otimes \Lambda^3 \otimes T^{r-i}
\twoheadrightarrow 
T^{\ell-1} \otimes \Lambda^3 \otimes T^{r-\ell}
\]
onto the $\ell$th factor induces an inclusion 
\begin{eqnarray}
\label{eqn:include_ker_dbar}
\ker \dbar_r \hookrightarrow T^{\ell-1} \otimes \Lambda^3 \otimes T^{r-\ell}.
\end{eqnarray}
\end{cor}

\begin{rem}
In general (as follows from Theorem \ref{THM}, for example) the inclusion (\ref{eqn:include_ker_dbar}) is not an isomorphism. For instance, for $r=3$, we can choose any $X_1 \in \Lambda^3 \otimes T^2$ and then set $X_2:=  -\alpha_{1;2} (X_1)$ and $X_3 := -\alpha_{1;3} (X_1)$. This corresponds to an element of $\ker \dbar_3$ if and only if 
$ 
\alpha_{2;3} (X_2) = - X_3.
$  
This condition is not satisfied for every $X_1$. The additional condition can be considered as a `cocycle' condition.
\end{rem}

However, for $r=2$, there is no `cocycle' condition, so that the inclusion (\ref{eqn:include_ker_dbar}) gives  an isomorphism $\ker \dbar_2 \cong  \Lambda^3 \otimes \Lambda^1$, taking $\ell =1$. This can be strengthened to take into account the $\sym_2$-equivariance: 

\begin{thm}
\label{thm:r=2}
For $r=2$, there is a $\sym_2$-equivariant isomorphism:
\[
\ker \dbar_2 
\cong 
\Lambda^4 \boxtimes \sgn_2 
\ \oplus \ 
\schur_{(2,1^2)} \boxtimes \triv_2. 
\]
Moreover, up to non-zero scalar
\begin{enumerate}
\item 
 the underlying inclusion $\Lambda^4 \hookrightarrow \Lambda^3 \otimes \Lambda^1 \oplus \Lambda^1 \otimes \Lambda^3$ is given by the coproduct of $\Lambda^*$;
\item 
the underlying inclusion $\schur_{(2,1^2)}  \hookrightarrow \Lambda^3 \otimes \Lambda^1 \oplus \Lambda^1 \otimes \Lambda^3$ corresponds to the commutative diagram:
\[
\xymatrix{
&
\schur_{(2,1^2)}
\ar@{^(->}[ld]
\ar[rd] 
\ar@{}[d]
\\
\Lambda^3 \otimes \Lambda^1 
\ar[rr]^\tau
&&
\Lambda^1 \otimes \Lambda^3
}
\] 
associated to $\schur_{(2,1^2)} \subset \Lambda^3 \otimes \Lambda^1 $.
\end{enumerate}
\end{thm}

\begin{proof}
By the discussion preceding the statement and using the case $r=2$ of Proposition \ref{prop:X_i_ker_dbar}, we have an isomorphism 
\begin{eqnarray}
\label{eqn:iso_r=2}
\Lambda^3 \otimes \Lambda^1 \stackrel{\cong} {\rightarrow } \ker \dbar_2 \subset \Lambda^3 \otimes \Lambda^1 \ \oplus \ \Lambda^1 \otimes \Lambda^3, 
\end{eqnarray}
where the composite is given by $X\mapsto (X, - \alpha X)$. 

The action of $\sym_2$ on $\ker \dbar_2$ is induced by the transposition of tensor factors $\Lambda^3 \otimes \Lambda^1 
\stackrel{\tau}{\rightarrow}
\Lambda^1 \otimes \Lambda^3$. Hence, under the isomorphism (\ref{eqn:iso_r=2}),  this corresponds to the automorphism $X \mapsto - (\tau \circ \alpha) X$ of $\Lambda^3 \otimes \Lambda^1$. By Proposition \ref{prop:new_tau_alpha}, this acts via the diagonal matrix $(-1, 1)$ (using the identification of Lemma \ref{lem:schur_Lambda3_otimes_Lambda1}, as in  Proposition \ref{prop:new_tau_alpha}).
 This proves the first statement and also identifies the inclusion of $\schur_{(2,1^2)}$. 

The identification of the embedding of $\Lambda^4$ then follows from the graded cocommutativity of the coproduct of $\Lambda^*$. 
\end{proof}

\subsection{A lower bound for $\ker \dbar_r$}
\label{subsect:lower}

Generalizing the case $r=2$, we establish a lower bound for $\ker \dbar_r$ for $r\geq 2$. This will later be shown to coincide with $\ker \dbar_r$.

\begin{prop}
\label{prop:lower_bound_ker}
For $r \geq 2$, there is a $\sym_r$-equivariant inclusion:
\[
\Lambda^{r+2}  \boxtimes \sgn_r \   \oplus \  \schur_{(r, 1^2)}  \boxtimes \triv_r
\hookrightarrow \ker \dbar_r.
\]
\end{prop}

\begin{proof}
We first observe that the terms $\Lambda^{r+2}  \boxtimes \sgn_r$ and $\schur_{(r, 1^2)} \boxtimes \triv_r$ are simple and non-isomorphic, hence it suffices to exhibit the appropriate inclusions for the individual factors. 

For $\lambda \in \{ (1^{r+2}), (r, 1^2) \}$, we have the following useful identification
$$
\mathrm{Nat} (\schur_\lambda, T^{i-1} \otimes \Lambda^3  \otimes T^{ r-i}) = \rat 
$$
for each $1 \leq i \leq r$, that is seen as follows (noting that it suffices to show this for $i=1$). Observe that the Young diagram for  $(1^{r+2})$ is obtained from that for $(1^3)$ by adding $r-1$ boxes to the first (and unique) column; likewise, the Young diagram for $(r, 1^2)$ is obtained from that for $(1^3)$ by adding $r-1$ boxes to the first row. The assertion follows by  using Pieri's rule in conjunction with this observation, together with Proposition \ref{prop:nat_schur}.

Let us fix inclusions (unique up to non-zero scalar multiple, by the above) 
\begin{eqnarray*}
\phi_1 &:& \Lambda^{r+2}  \hookrightarrow  \Lambda^3 \otimes T^{r-1} \\
\psi_1 &: &\schur_{(r, 1^2)}  \hookrightarrow  \Lambda^3 \otimes T^{r-1},
\end{eqnarray*}
where the subscripts indicate that these correspond to $i=1$. To be concrete, we take $\phi_1$ to be the iterated coproduct for the exterior algebra $\Lambda^*$. For $\psi_1$, we first choose a (unique up to non-zero scalar multiple) inclusion $\schur_{(r, 1^2)} \hookrightarrow \Lambda^3 \otimes \Gamma^{r-1}$ (see Example \ref{exam:schur_embedding}), and then compose with the inclusion $\Lambda^3 \otimes \Gamma^{r-1} \hookrightarrow \Lambda^3 \otimes T^{r-1}$ induced by the canonical inclusion $\Gamma^{r-1} \subset T^{r-1}$.

Recall that  the group $\sym_r$ acts on the domain $(\Lambda^3 \otimes T^{r-1}) \uparrow _{\sym_{r-1}} ^{\sym_r}$ of $\dbar_r$. In particular, for any $\sigma \in \sym_r$, we have the restriction of the action of $\sigma$ to the direct summand $\Lambda^3 \otimes T^{r-1}$:
$$
\Lambda^3 \otimes T^{r-1} 
\stackrel{\sigma}{\rightarrow} 
T^{i-1} \otimes \Lambda^3  \otimes T^{ r-i}
\subset (\Lambda^3 \otimes T^{r-1}) \uparrow _{\sym_{r-1}} ^{\sym_r}
$$
where $i = \sigma^{-1} (1)$. (Note the (slightly abusive) usage of the notation $\sigma$.)

We claim that, for given $i$ and $\sigma \in \sym_r$ such that $i = \sigma^{-1} (1)$, the following composites:
\begin{eqnarray}
\label{eqn:phi_i}
&&\Lambda^{r+2} \stackrel{\phi_1}{\longrightarrow} \Lambda^3 \otimes T^{r-1} 
\stackrel{\sgn(\sigma)\sigma}{\longrightarrow }
T^{i-1} \otimes \Lambda^3  \otimes T^{ r-i}
\\
&&
\label{eqn:psi_i}
\schur_{(r, 1^2)}   \stackrel{\psi_1}{\longrightarrow} \Lambda^3 \otimes T^{r-1}\stackrel{\sigma}{\longrightarrow }
T^{i-1} \otimes \Lambda^3  \otimes T^{ r-i}
\end{eqnarray}
do not depend on the choice of $\sigma$. This follows easily from the explicit descriptions of $\phi_1$ and $\psi_1$ given above. 

Hence, for any such choice of $\sigma$, we may define $\phi_i$ to be the composite (\ref{eqn:phi_i}) and $\psi_i$ to be the composite (\ref{eqn:psi_i}). (This definition is imposed  by the isomorphism that we are trying to establish,  due to  the terms $\sgn_r$ and $\triv_r$ appearing in the statement.) By construction, these yield an inclusion
$$
\Lambda^r \oplus \schur_{(r,1^2)} 
\hookrightarrow 
(\Lambda^3 \otimes T^{r-1}) \uparrow _{\sym_{r-1}} ^{\sym_r}.
$$
To conclude, we must prove that they map to $\ker \dbar_r$, since the construction ensures that the $\sym_r$-action is correct (as remarked above).

Hence, by Proposition \ref{prop:X_i_ker_dbar}, it suffices to prove that, for any $i<j$, the following two diagrams anti-commute:
$$
\xymatrix{
& 
\Lambda^{r+2}
\ar[ld]_{\phi_i}
\ar[rd]^{\phi_j}
\ar@{}[d]|{-}
\\
T^{i-1} \otimes \Lambda^3  \otimes T^{ r-i}
\ar[rr]_{\alpha_{i;j}}
&&
T^{j-1} \otimes \Lambda^3  \otimes T^{ r-j}
}
$$
 and 
 $$
\xymatrix{
& 
\schur_{(r,1^2)} 
\ar[ld]_{\psi_i}
\ar[rd]^{\psi_j}
\ar@{}[d]|{-}
\\
T^{i-1} \otimes \Lambda^3  \otimes T^{ r-i}
\ar[rr]_{\alpha_{i;j}}
&&
T^{j-1} \otimes \Lambda^3  \otimes T^{ r-j}.
}
$$
 
The case $i=1$ and $j=2$ is a straightforward extension of Theorem \ref{thm:r=2}. (For the case $\Lambda^{r+2}$, one uses the coassociativity of the coproduct of $\Lambda^*$ (denoted $\Delta$ below) to give the factorization of $\phi_1$ as the composition:
$$
\Lambda^{r+2} \stackrel{\Delta}{\longrightarrow} \Lambda^4 \otimes \Lambda^{r-2} \stackrel{\Delta \otimes \Delta}{\longrightarrow} \Lambda^3 \otimes \Lambda^1 \otimes (\Lambda^1)^{\otimes r-2} \cong \Lambda^3 \otimes T^{r-1}. 
$$
This makes the relationship with the case $r=2$ transparent.)

Finally, for the general case $i<j$, we exploit the relationship between the different $\alpha_{i;j}$ in terms of the $\sym_r$-action to reduce to the case $i=1$, $j=2$. This is based on the following commutative diagram, using restrictions of the (left) action of $\sym_r$ on   $(\Lambda^3 \otimes T^{r-1}) \uparrow _{\sym_{r-1}} ^{\sym_r}$, as above:
\[
\xymatrix{
\Lambda^3 \otimes T^{r-1}
\ar[r]^{(12)}
\ar[d]_{(1i)}
&
T^1 \otimes \Lambda^3 \otimes T^{r-2}
\ar[d]^{(1i)(12)(j2)} 
\\
T^{i-1} \otimes \Lambda^3 \otimes T^{r-i}
\ar[r]
_{(ij)}
&
T^{j-1} \otimes \Lambda^3 \otimes T^{j-i}
}
\]
(the commutativity is checked directly in $\sym_r$). 

One has an analogous commutative diagram in which the horizontal maps are replaced by $\alpha_{1;2}$ and $\alpha_{i;j}$ respectively:
\[
\xymatrix{
\Lambda^3 \otimes T^{r-1}
\ar[r]^{\alpha_{1;2}}
\ar[d]_{(1i)}
&
T^1 \otimes \Lambda^3 \otimes T^{r-2}
\ar[d]^{(1i)(12)(j2)} 
\\
T^{i-1} \otimes \Lambda^3 \otimes T^{r-i}
\ar[r]
_{\alpha_{i;j}}
&
T^{j-1} \otimes \Lambda^3 \otimes T^{j-i}.
}
\]

Composing the respective anti-commutative triangles for the case $i=1$, $j=2$ with this commutative square yields the general case $i<j$, by the definition of $\phi_i$ (respectively $\psi_i$) in terms of $\phi_1$ (resp. $\psi_1$). In the case $\Lambda^{r+2}$, this relies upon the fact that {\em both} the vertical maps are induced by permutations of sign $-1$.
\end{proof}

\begin{rem}
The case $r=3$ of this result is established explicitly in the (easy part of the) proof of Theorem \ref{thm:case_r=3}. Moreover, this can be used to give an alternative  proof of Proposition \ref{prop:lower_bound_ker}.
\end{rem}

\section{The proof of Theorem \ref{THM}}
\label{sect:main}

The purpose of this section is to prove Theorem \ref{THM}, based on the analysis of $\ker \dbar_r$ (for $r\geq 2$) by applying Proposition \ref{prop:H_1_by_dbar}. The (calculational) proof of the  initial step of the inductive proof is postponed until Section \ref{sect:case_r3} (that proof is entirely independent of the material in this section). 
 
The proof exploits the lemmas concerning Schur functors given in Section \ref{subsect:preparatory}.
 
\subsection{Preparatory lemmas}
\label{subsect:preparatory}

Consider the exterior power functors $\Lambda^n \cong \schur_{(1^n)} $, for $n \in \nat$. The coproduct provides the inclusion $\Lambda^{n+1}  \hookrightarrow \Lambda^n  \otimes \Lambda^1 $;  up to non-zero scalar multiple, this is the unique non-trivial natural transformation of this form.  On applying $-\otimes \Lambda^1 $, one therefore obtains $\Lambda^{n+1} \otimes \Lambda^1  \hookrightarrow \Lambda^n  \otimes T^2$. Post-composing with $\id_{\Lambda^n } \otimes \tau$ gives a second, distinct embedding. 

Up to non-zero scalar multiple, there is a unique non-trivial map from $\Lambda^{n+2}$ to $\Lambda^n \otimes  T^2$, so that $\Lambda^{n+2}$ is canonically a subfunctor of the codomain. These subobjects of $\Lambda^n \otimes T^2$ are related by:

\begin{lem}
\label{lem:intersect_lambda}
For $n \geq 1$, the intersection of $\Lambda^{n+1}\otimes \Lambda^1 $ and $(\id \otimes \tau) \big(\Lambda^{n+1} \otimes \Lambda^1 \big)$ in  $\Lambda^n \otimes T^2$ is $\Lambda^{n+2}$.
\end{lem}

\begin{proof}
Write $\sym_{n+1}\subset \sym_{n+2}$ for the subgroup fixing $n+2$ and $\sym_{n+1}'\subset \sym_{n+2}$ for that fixing $n+1$. Considering $\Lambda^{n+1} \otimes \Lambda^1 $ and $(\id \otimes \tau) \big(\Lambda^{n+1} \otimes \Lambda^1  \big)$ as subobjects of $T^{n+2}$, the former identifies as the subspace of elements on which $\sigma \in \sym_ {n+1}$ acts  by multiplication by $\sgn (\sigma)$ and the latter similarly, with respect to $\sym_{n+1}'$. Since the subgroups $\sym_{n+1}$ and $\sym_{n+1}'$ generate $\sym_{n+2}$, it follows that any $\rho \in \sym_{n+2}$ acts on the intersection by multiplication by $\sgn(\rho)$. The result follows.
\end{proof}

We now consider the family of functors $\schur_{(n,1^2)}$, assuming that $n\geq 2$. 
One has the embedding
\[
\schur_{(n, 1^2)}
\hookrightarrow 
\Lambda^3  \otimes \Gamma^{n-1}
\hookrightarrow 
\Lambda^3 \otimes \Gamma^{n-2}  \otimes \Gamma^1,
\]
where the first map is given by Example \ref{exam:schur_embedding} and the second map by the coproduct of $\Gamma^*$. 

Applying the functor $-\otimes \Gamma^1$, one obtains the  inclusion 
\begin{eqnarray}
\label{eqn:include_n1^2_otimes_1}
\schur_{(n, 1^2)} \otimes \Gamma^1 \subset \Lambda^3 \otimes \Gamma^{n-2}  \otimes T^2.
\end{eqnarray}
 Post-composing with  $(\mathrm{Id} \otimes \tau)$ 
gives the inclusion $(\mathrm{Id} \otimes \tau)\big( \schur_{(n, 1^2)} \otimes \Gamma^1)  \subset  \Lambda^3  \otimes \Gamma^{n-2}  \otimes T^2$. 

Using the canonical inclusion $\Gamma^{n-2} \subset T^{n-2}$, we may consider the functors $\schur_ {(n,1^2)} \otimes \Gamma^1 $ and $(\mathrm{Id} \otimes \tau)\big(\schur_{(n, 1^2)} \otimes \Gamma^1 \big)$ as lying in $\Lambda^3 \otimes T^{n}$. 

The following result is a counterpart of Lemma \ref{lem:intersect_lambda}:

\begin{lem}
\label{lem:intersect_schur}
For $n \geq 2$, the intersection of $\schur_ {(n,1^2)} \otimes \Gamma^1 $ and $(\mathrm{Id} \otimes \tau)\big(\schur_{(n, 1^2)} \otimes \Gamma^1 \big)$ in $ \Lambda^3 \otimes T^n$ is isomorphic to 
$\schur_{(n+1,1^2)}$. 
\end{lem}

\begin{proof}
One first observes that the intersection, denoted here simply by $\bigcap$, lies in $\Lambda^3 (V) \otimes \Gamma^n (V)$.
This follows from an argument similar to that used in the proof of Lemma \ref{lem:intersect_lambda}.

It then remains to show (using the identification of $\schur_{(n+1,1^2)}$ given in Example \ref{exam:schur_embedding}) that this intersection lies in the kernel of 
$\diff$. 

By construction, there is a commutative diagram:
\[
\xymatrix{
\bigcap 
\ar@{^(->}[r]
\ar@{^(->}[d]
&
\Lambda^3  \otimes \Gamma^n 
\ar@{^(->}[d]
\\
\schur_{(n,1^2)} \otimes \Gamma^1
\ar@{^(->}[r]
\ar[rd]_0 
&
\Lambda^3  \otimes T^n
\ar[d]^{\mu \otimes \mathrm{Id}}
\\
&
\Lambda^4  \otimes T^{ n-1},
}
\]
in which $\mu : \Lambda^3 \otimes \Lambda^1 \rightarrow \Lambda^4 $ is the product and the top square is given by the definition of the intersection $\cap$ and the above, using the embedding (\ref{eqn:include_n1^2_otimes_1}).
 
The labelled map is zero by Lemma \ref{lem:ker_coker_diff}, using the characterization of $\diff$ given by Lemma \ref{lem:diff_characterization}. This implies that $\bigcap$ lies in the kernel of the right hand vertical composite. Hence, using  Example \ref{exam:schur_embedding}, $\bigcap$ lies in $\ker \diff \cong \schur_{(n+1,1^2)}$. To conclude, it remains to show that this is an equality (this is equivalent to establishing the non-triviality of $\bigcap$).

By Pieri's rule,  $ \Lambda^3 \otimes \Gamma^{n-2} \otimes T^2$
 contains a unique composition factor of $\schur_{(n+1,1^2)}$; this occurs in both $\schur_ {(n,1^2)}  \otimes \Gamma^1 $ and $(\mathrm{Id} \otimes \tau)\big(\schur_{(n, 1^2)} \otimes \Gamma^1\big)$, hence occurs in the intersection, thus concluding the proof. 
\end{proof}

\subsection{The proof of Theorem \ref{THM} (assuming Theorem \ref{thm:case_r=3})}
The following proof exploits the description of $\ker \dbar_r$ given in Proposition \ref{prop:X_i_ker_dbar}.

\begin{proof}
The cases $r\in \{ 1, 2\}$ have been established in Example \ref{exam:r=1} and  Theorem \ref{thm:r=2} respectively. The case $r=3$ is  a (somewhat technical) calculation;  this has been postponed until  Section \ref{sect:case_r3}, where it is proved as Theorem \ref{thm:case_r=3}.

Hence, by induction upon $r$,  we may suppose that $r >3$ and that the result has been proved for $r-1$. It remains to  prove the inductive step. 

Proposition \ref{prop:lower_bound_ker} established the $\sym_r$-equivariant natural inclusion $
\Lambda^{r+2}\boxtimes \sgn_r \   \oplus \  \schur_{(r, 1^2)}  \boxtimes \triv_r \subset \ker (\dbar_r)$, 
hence it suffices only to consider the underlying functor, neglecting the $\sym_r$-action. 

Based upon Proposition \ref{prop:X_i_ker_dbar}, Corollary \ref{cor:ell_inclusion} with $\ell =1$ gives the inclusion 
\[
\ker \dbar_r \subset \Lambda^3 \otimes T^{r-1}.
\]

Using Proposition \ref{prop:X_i_ker_dbar}, we claim that there are inclusions:
\begin{eqnarray}
\label{eqn:first_face}
\ker \dbar_r &\subset & \ker \dbar_{r-1} \otimes \Lambda^1 
\\
\label{eqn:second_face}
\ker \dbar_r &\subset & (\id_{\Lambda^3 \otimes T^{\otimes r-3}} \otimes \tau ) \big( \ker \dbar_{r-1} \otimes \Lambda^1  \big), 
\end{eqnarray}
where both terms are considered as subobjects of $\Lambda^3 \otimes T^{r-1}$ and $\ker \dbar_{r-1}$ is considered as a subobject of $\Lambda^3 \otimes T^{r-2}$, again by 
Corollary \ref{cor:ell_inclusion} with $\ell =1$. The $\tau$ corresponds to the action of $\sym_2$ by place permutations on $T^2$, the last two tensor factors of $T^{r-1}$.

To establish this claim, it is useful to visualize the statement of Proposition \ref{prop:X_i_ker_dbar} by considering that $i \in \{1, \ldots, r\}$ indexes the vertices of a $r-1$-simplex, analogously to the geometric visualization proposed in Section \ref{subsect:geometric_vis}. Then the inclusion (\ref{eqn:first_face}) is obtained by  restricting to the codimension one face given by $i \neq r$, using Proposition \ref{prop:X_i_ker_dbar} for $r-1$ to identify $\ker \dbar_{r-1}$. The inclusion (\ref{subsect:geometric_vis}) is obtained similarly, this time using the face $i \neq r-1$. Using the equivariance property of $\dbar_r$, one obtains the stated inclusion by exploiting the transposition of tensor factors $\tau$.

Now, by the inductive hypothesis, there is a natural isomorphism:
\[
\ker \dbar_{r-1}  \otimes \Lambda^1
\cong 
\Lambda^{r+1} \otimes \Lambda^1  \ \oplus \ \schur_{(r-1,1^2)} \otimes \Lambda^1.
\]
Hence, putting the information from the two faces (\ref{eqn:first_face}) and (\ref{eqn:second_face}) together gives
\begin{eqnarray}
\label{eqn:include_cap}
\ker \dbar_r 
\subset 
\big(
\Lambda^{r+1} \otimes \Lambda^1 \ \oplus \ \schur_{(r-1,1^2)}\otimes \Lambda^1 
\big) 
\cap 
(\id \otimes \tau) \big(
\Lambda^{r+1}\otimes \Lambda^1  \ \oplus \ \schur_{(r-1,1^2)} \otimes \Lambda^1 
\big) 
\end{eqnarray}
where the intersection is formed as subobjects of $\Lambda^3 \otimes T^{r-1}$.

The Pieri rule gives the natural isomorphisms:
\begin{eqnarray*}
\Lambda^{r+1}\otimes \Lambda^1 
&\cong &
\Lambda^{r+2} \oplus \schur_{(2, 1^r)}  \\
 \schur_{(r-1,1^2)} \otimes \Lambda^1 
 &\cong &
 \schur_{(r,1^2)} \oplus \schur_{(r-1,2, 1)} \oplus \schur_{(r-1,1^3)}.
\end{eqnarray*}
In particular, since $r>3$ by hypothesis, $\Lambda^{r+1} \otimes \Lambda^1$ and $ \schur_{(r-1,1^2)} \otimes \Lambda^1$ have no composition factors in common. 
The analogous statement holds after applying the isomorphism $(\id \otimes \tau)$ to one of the factors. 
 This implies that the right hand side of (\ref{eqn:include_cap}) is equal to
\[
\Big((\Lambda^{r+1} \otimes \Lambda^1 ) \cap (\mathrm{Id} \otimes \tau)\big( \Lambda^{r+1} \otimes \Lambda^1  \big)\Big) 
\ \oplus \ 
\Big(
(\schur_{(r-1,1^2)} \otimes \Lambda^1) \cap (\mathrm{Id} \otimes \tau)\big(\schur_{(r-1,1^2)} \otimes \Lambda^1 \big)
\Big).
\]
By Lemmas \ref{lem:intersect_lambda} and  \ref{lem:intersect_schur} for the respective terms, the latter identifies with 
\[
\Lambda^{r+2}
\oplus
\schur_{(r,1^2)}, 
\]
giving an upper bound for $ \ker \dbar_r$. Since this coincides with the lower bound (neglecting the $\sym_r$-action) by Proposition \ref{prop:lower_bound_ker}, this  completes the proof of the inductive step and hence of the Theorem.
\end{proof}

\begin{rem}
This argument does not go through for $r=3$, since in that case $\Lambda^{4} \otimes \Lambda^1$ and $ \schur_{(2,1^2)} \otimes \Lambda^1 $ have a common  factor of $\schur_{(2,1^3)}$.
\end{rem}

 \section{The case $r=3$}
 \label{sect:case_r3}

This section  completes the proof of   Theorem \ref{THM} by establishing the postponed case $r=3$ that formed the initial step of the induction. 
 (This proof is independent of the material of Section \ref{sect:main}.)

\begin{thm}
\label{thm:case_r=3}
There is a $\sym_3$-equivariant natural isomorphism 
\[
\ker \dbar_3
\cong 
\Lambda^{5} \boxtimes \sgn_3\   \oplus \  \schur_{(3, 1^2)} \boxtimes \triv_3.
\]
\end{thm}

The proof uses Proposition \ref{prop:X_i_ker_dbar} for the case $r=3$. The relevant morphisms (see Notation \ref{nota:alpha_ij}) fit into  the (non-commutative) diagram:
\[
\xymatrix{
& 
\Lambda^1 \otimes \Lambda^1 \otimes \Lambda^3 
\ar[ldd]_{\alpha_{1;3}^{-1}} 
\\
\\
\Lambda^3 \otimes \Lambda^1 \otimes \Lambda^1 
\ar[rr]
_{\alpha_{1;2}}
&&
\Lambda^1 \otimes \Lambda^3 \otimes \Lambda^1. 
\ar[uul]_{\alpha_{2;3}} 
}
\]

The cocycle condition arising from Proposition \ref{prop:X_i_ker_dbar} can be restated as

\begin{prop}
\label{prop:cocycle_r=3}
The kernel of $\dbar_3$ is isomorphic to the equalizer of 
\[
\xymatrix{
\Lambda^3 \otimes \Lambda^1 \otimes \Lambda^1 
\ar@<.5ex>[rr]^\id 
\ar@<-.5ex>[rr]_{- (\alpha_{1;3}^{-1} \circ \alpha_{2;3} \circ \alpha_{1;2})}
&
\quad 
&
\Lambda^3 \otimes \Lambda^1 \otimes \Lambda^1. 
}
\]
\end{prop}

The composite $\alpha_{1;3}^{-1} \circ \alpha_{2;3} \circ \alpha_{1;2}$ is an automorphism of $\Lambda^3 \otimes \Lambda^1 \otimes \Lambda^1 $.  
This will be analysed using the following automorphism:

\begin{nota}
\label{nota:beta}
Denote by $\beta$ the automorphism of $\Lambda^3 \otimes \Lambda^1 \otimes \Lambda^1$ given by the composite
\[
\Lambda^3 \otimes \Lambda^1 \otimes \Lambda^1
\stackrel{\tau \circ \alpha \otimes \id}{\longrightarrow}
\Lambda^3 \otimes \Lambda^1 \otimes \Lambda^1
\stackrel{\id \otimes \tau} {\longrightarrow }
\Lambda^3 \otimes \Lambda^1 \otimes \Lambda^1.
\]
\end{nota}

\begin{rem}
\label{rem:beta}
The morphism $\beta$ also identifies as the composite of $\alpha \otimes \id : \Lambda^3 \otimes \Lambda^1 \otimes \Lambda^1 \rightarrow \Lambda^1 \otimes \Lambda^3 \otimes \Lambda^1$ with the cyclic place permutation $\Lambda^1 \otimes \Lambda^3 \otimes \Lambda^1 \stackrel{\mathrm{cyclic}}{\rightarrow } \Lambda^3 \otimes \Lambda^1 \otimes \Lambda^1 $ (this corresponds to switching factors $\Lambda^1 \otimes F\rightarrow F \otimes \Lambda^1$, where $F:=\Lambda^3 \otimes \Lambda^1$.)
\end{rem}

\begin{lem}
\label{lem:alpha_comp_versus beta}
The composite $\alpha_{1;3}^{-1}\circ \alpha_{2;3} \circ \alpha_{1;2}$ is equal to the automorphism
 $\beta^{\circ 3}$ of $\Lambda^3 \otimes \Lambda^1 \otimes \Lambda^1$.
\end{lem}

\begin{proof}
This is a direct verification, based on the construction of the morphisms $\alpha_{1;3}$,   $\alpha_{2;3}$,  $\alpha_{1;2}$ from $\alpha$. Namely, using Remark \ref{rem:beta},
one reduces to checking that the cyclic permutation in the definition of $\beta$ always ensures that the $\alpha$ is applied to the correct choice of factors, in the correct order. 
\end{proof}

Using Lemma \ref{lem:alpha_comp_versus beta}, Proposition \ref{prop:cocycle_r=3} can be restated as:

\begin{prop}
\label{prop:restate_cocycle_r=3}
The kernel of $\dbar_3$ is equal to  the equalizer of $\id$ and $(- \beta)^{\circ 3}$.
\end{prop} 

 We thus need to understand the automorphism $- \beta$ of $\Lambda^3 \otimes \Lambda^1 \otimes \Lambda^1$. This  identifies explicitly using Corollary \ref{cor:new_tau_alpha}:

\begin{lem}
\label{lem:-beta}
The automorphism $- \beta$ of $\Lambda^3 \otimes \Lambda^1 \otimes \Lambda^1$ evaluated on $V$ is given for $x,y,z,s,t\in V$ by 
 $$- \beta  (x \wedge y \wedge z \otimes s \otimes t ) = 
\frac{1}{2} \big( x \wedge y \wedge z  \otimes t \otimes s + x \wedge y \wedge s \otimes t \otimes z - x \wedge z \wedge s \otimes t \otimes   y + y \wedge z \wedge s \otimes t  \otimes x \big).$$ 
\end{lem}

To analyse the automorphism $- \beta$ by exploiting the representation theory, one can use the canonical  decomposition of $\Lambda^3  \otimes \Lambda^1  \otimes \Lambda^1$ that is given by Pieri's rule:
\begin{eqnarray}
\label{eqn:pieri_311}
\Lambda^3 \otimes \Lambda^1 \otimes \Lambda^1
\cong 
\Lambda^5 \oplus \schur_{(3,1^2)}  \oplus \schur_{(2^2,1)} \oplus F_{(2,1^3)}
\end{eqnarray}
where $F_{(2,1^3)}$ is non-canonically isomorphic to 
$\schur_{(2,1^3)}^{\oplus 2}$.

By Schur's lemma, we have the following analogue of Lemma \ref{lem:schur_Lambda3_otimes_Lambda1}:

\begin{lem}
\label{lem:endo_auto}
There is an isomorphism of algebras
\begin{eqnarray*}
\mathrm{End}(\Lambda^3 \otimes \Lambda^1\otimes \Lambda^1)
&\cong & 
\mathrm{End}(\Lambda^5) 
\prod 
\mathrm{End}(\schur_{(3,1^2)} ) 
\prod
\mathrm{End}(\schur _{(2^2,1)}) 
\prod 
\mathrm{End}(F_{(2,1^3)}). 
\end{eqnarray*}

A choice of isomorphism $F_{(2,1^3)} \cong \schur_{(2,1^3)}^{\oplus 2}$ induces the isomorphism
\begin{eqnarray}
\label{eqn:iso_depends_F_iso}
\mathrm{End}(\Lambda^3 \otimes \Lambda^1\otimes \Lambda^1)
\cong
\rat_{(1^5)} \prod \rat_{(3,1^2)}  \prod \rat_{(2^2,1)} \prod M_2 (\rat).
\end{eqnarray}
\end{lem}

\begin{rem}
The fact that the decomposition (\ref{eqn:pieri_311}) is not multiplicity-free is the source of complexity in the proof of Theorem \ref{thm:case_r=3}. In particular,  we stress that 
the isomorphism $\mathrm{End}(F_{(2,1^3)}) \cong M_2 (\rat)$ is not canonical.
\end{rem}

We analyse the isomorphism (\ref{eqn:pieri_311}) in two different ways, adapted respectively to considering the action of $\tau \circ \alpha \otimes \id$ and of $\id \otimes \tau$, the two isomorphisms appearing in the construction of $\beta$ in Notation \ref{nota:beta}.

For the first, starting from the decomposition $\Lambda^3 \otimes \Lambda^1 \cong \Lambda^4 \oplus \schur _{(2,1^2)}$, one uses the identifications given by the Pieri rule
\begin{eqnarray*}
\Lambda^4 \otimes \Lambda^1 & \cong & \Lambda^5 \oplus \schur_{(2, 1^3)} \\
\schur_{(2,1^2)} \otimes \Lambda^1 & \cong & \schur_{(3, 1^2)} \oplus \schur_{(2^2,1)} \oplus \schur _{(2,1^3)}.
\end{eqnarray*}
In particular, this provides an isomorphism $F_{(2,1^3)} \cong \schur_{(2,1^3)}^{\oplus 2}$, where $F_{(2,1^3)}$ is as in (\ref{eqn:pieri_311}). Then, by Proposition \ref{prop:new_tau_alpha}, one has the following, in which $(1,-1)$ denotes the diagonal matrix in $M_2 (\rat)$ with respect to the chosen decomposition:

\begin{lem}
\label{lem:tau_alpha_otimes_id}
Using the above isomorphism $F_{(2,1^3)} \cong \schur_{(2,1^3)}^{\oplus 2}$, with respect to the corresponding  isomorphism (\ref{eqn:iso_depends_F_iso}), 
the endomorphism $\tau\circ \alpha \otimes \id \in \mathrm{End}(\Lambda^3 \otimes \Lambda^1\otimes \Lambda^1)$ identifies as $(1, -1, -1, (1, -1))$.
\end{lem}

To study the action of $\id \otimes \tau$, one uses the $\sym_2$-equivariant decomposition $\Lambda^1 \otimes \Lambda^1 \cong \schur_{(2)} \boxtimes \triv_2  \oplus \Lambda^2 \boxtimes \sgn_2$ (for the place permutation action on the left hand side) together with the isomorphisms given by the Pieri rule:
\begin{eqnarray*}
\Lambda^3 \otimes \schur_{(2)} & \cong & \schur_{(3,1^2)} \oplus \schur _{(2,1^3)} \\
\Lambda^3 \otimes \Lambda^2 & \cong & \schur_{(2^2,1)} \oplus \schur _{(2,1^3)} \oplus \Lambda^5.
\end{eqnarray*}

This provides a different isomorphism $F_{(2,1^3)} \cong \schur_{(2,1^3)}^{\oplus 2}$. 
Then, analogously to Lemma \ref{lem:tau_alpha_otimes_id}, one has: 

\begin{lem}
\label{lem:id_otimes_tau}
Using this isomorphism $F_{(2,1^3)} \cong \schur_{(2,1^3)}^{\oplus 2}$, with respect to the corresponding  isomorphism (\ref{eqn:iso_depends_F_iso}), 
the endomorphism $\id \otimes \tau \in \mathrm{End}(\Lambda^3 \otimes \Lambda^1\otimes \Lambda^1)$ identifies as $(-1, 1, -1, (1, -1))$.
\end{lem}

Since there are unique composition factors of $\Lambda^5$,  $\schur_{(3,1^2)}$, and  $\schur_{(2^2,1)}$ in $\Lambda^3 \otimes \Lambda^1 \otimes \Lambda^1$, Lemmas \ref{lem:tau_alpha_otimes_id} and \ref{lem:id_otimes_tau} imply:

\begin{lem}
\label{lem:-beta_mult_one}
The automorphism $- \beta$ restricts to 
\begin{enumerate}
\item 
$\id$ on the summand $\Lambda^5 \oplus \schur_{(3,1^2)}$; 
\item 
$- \id$ on the summand $\schur_{(2^2,1)}$.
\end{enumerate}
\end{lem}

To prove Theorem \ref{thm:case_r=3}, it remains to show that $(-\beta)^{\circ 3}$ has no non-zero fixed points for the factor $F_{(2,1^3)}$; this restriction can be considered as an element of $ M_2 (\rat)$ via Lemma \ref{lem:endo_auto} (by choosing an isomorphism $F_{(2,1^3)} \cong \schur_{(2,1^3)}^{\oplus 2}$). It is sufficient to calculate the eigenvalues of this action (which do not depend upon the choice of isomorphism).

Now, the decompositions corresponding to the two composition factors $\schur_{(2,1^3)}$ in $\Lambda^3 \otimes \Lambda^1 \otimes \Lambda^1$ 
 that are used in Lemmas \ref{lem:tau_alpha_otimes_id} and \ref{lem:id_otimes_tau} do not coincide (that they must be different is a consequence of the calculations below), so one cannot directly read off the action of $- \beta$ on the corresponding isotypical component. One can, however, deduce the determinant:

\begin{lem}
\label{lem:determinant}
The automorphism $- \beta$ restricted to $F_{(2,1^3)}$, considered as an element of $M_2 (\rat)$ via Lemma \ref{lem:endo_auto},   
 has determinant $1$.
\end{lem}

\begin{proof}
Using Lemmas \ref{lem:tau_alpha_otimes_id} and \ref{lem:id_otimes_tau}, one sees that the determinant of $\beta$ is $1$, using that the diagonal matrix $\mathrm{diag} (1, -1)$ has determinant $-1$. Thus $- \beta$ also has determinant $1$, since $- \id = \mathrm{diag} (-1,-1)$ has determinant one. 
\end{proof}

To determine the eigenvalues, it thus suffices to  calculate the trace.  For this, it is convenient to pass to the associated representations of $\sym_5$ and their underlying vector spaces, using the Schur correspondence that is given by the Schur functor construction, with inverse given by cross effects.

The Schur functor $\Lambda^3 \otimes \Lambda^1\otimes \Lambda^1$ corresponds to the $\sym_5$-representation 
\begin{eqnarray}
\label{eqn:sym_5}
(\sgn_3 \boxtimes \triv_1  \boxtimes \triv_1) \uparrow_{\sym_3 \times \sym_1 \times \sym_1}^{\sym_5}
\cong
\sgn_5 
\oplus 
S_{(3,1^2)} \oplus S_{(2^2,1)} \oplus S_{(2,1^3)}^{\oplus 2}, 
\end{eqnarray}
This module has dimension $20$; the dimensions of $\sgn_5$, $S_{(3,1^2)}$,  $S_{(2^2,1)}$, and  $S_{(2,1^3)}$ are respectively $1$, $6$, $5$, and $4$.
 The automorphism $-\beta$ induces an automorphism of the $\sym_5$-module given by (\ref{eqn:sym_5}), and hence of its underlying $\rat$-vector space, $\rat^{\oplus 20}$. 

\begin{lem}
\label{lem:m}
The element of   $M_{20}(\rat)$ associated to $-\beta$ has trace $0$.
\end{lem}

\begin{proof}
The vector space underlying (\ref{eqn:sym_5}) has basis given by  the elements $(i \wedge j \wedge k) \otimes s \otimes t$, where $i, j, k, s,t$ are pairwise distinct elements of $\{1, \ldots ,5 \}$ with $i<j<k$. (This basis can be ordered lexicographically.)

By inspection using Lemma \ref{lem:-beta} and its analogue for the $\sym_5$-representations (given by considering cross-effects), the associated matrix is zero on the diagonal, hence the trace is zero.
\end{proof}

From this, one deduces the following counterpart of Lemma \ref{lem:determinant}

\begin{lem}
\label{lem:trace}
The automorphism $- \beta$ restricted to $F_{(2,1^3)} \cong \schur_{(2,1^3)} ^{\oplus 2}$, considered as an element of $M_2 (\rat)$ via Lemma \ref{lem:endo_auto},   
 has trace $-\frac{1}{2}$ (which is independent of the choice of the isomorphism).
\end{lem}

\begin{proof}
We first calculate the trace of the restriction of $-\beta$ (considered as an element of $M_{20}(\rat)$ as in Lemma \ref{lem:m}) to the direct summand $S_{(2,1^3)} ^{\oplus 2}$ (which has underlying vector space of dimension $8$). Using the decomposition (\ref{eqn:sym_5}) and Schur's lemma, $-\beta$ has block diagonal form. Moreover, by Lemma \ref{lem:-beta_mult_one}, it acts as $\id$ on $\sgn_5 \oplus S_{(3,1^2)}$ (which has dimension $7$) and by $-\id$ on $S_{(2^2,1)}$ (which has dimension $5$). Using the block diagonal form and Lemma \ref{lem:m}, it follows that the restriction of $- \beta$ to $S_{(2,1^3)} ^{\oplus 2}$ (an element of $M_8 (\rat)$) has trace $-2$. 

By Schur's Lemma, this element of $M_8 (\rat)$ is obtained from $\mathrm{End}_{\sym_5} (S_{(2,1^3)} ^{\oplus 2}) \cong M_2 (\rat)$ by the morphism of algebras $M_2 (\rat) \rightarrow M_8 (\rat)$ induced by the functor $- \otimes \rat^{\oplus 4}$ on vector spaces. (The trace is independent of choices made.) The result follows.
\end{proof}

\begin{proof}[Proof of Theorem \ref{thm:case_r=3}]
By Proposition \ref{prop:restate_cocycle_r=3},  $\ker \dbar_3$ is  the equalizer of $\id$ and $(-\beta)^{\circ 3}$. The behaviour of this on the single-multiplicity composition factors is given by Lemma \ref{lem:-beta_mult_one}. To conclude, we need to show that $\id - (-\beta)^{\circ 3}$ is injective restricted to $F_{(2,1^3)}$. 

Using the trace and determinant, from Lemmas \ref{lem:determinant} and \ref{lem:trace}, one deduces that the automorphism $- \beta$ restricted to $F_{(2,1^3)}$, considered as an element of $M_2 (\rat)$ via Lemma \ref{lem:endo_auto},  has eigenvalues (in $\mathbb{C}$) 
$ 
\frac{1}{4}(-1  \pm \sqrt{15} i).
$ 
This implies that $1$ is not an eigenvalue of $(-\beta)^{\circ 3}$, whence the result.
\end{proof}

\appendix 

\section{The relation with outer functors on free groups}
\label{sect:motivation}

The purpose of this section is to provide a few more details on the motivation for this work that was outlined in the Introduction. This should also help make the application to the proof of Conjecture \ref{COR:GH} slightly more transparent.

Let $\gr$ be the category of finitely-generated free groups and $\f (\gr\op; \rat)$ the category of functors from $\gr \op$ to $\rat$-vector spaces. This is equipped with the pointwise tensor product, so that $(F_1 \otimes F_2) (G) = F_1 (G) \otimes F_2 (G)$, for functors $F_1$, $F_2$ and $G$ a free group.

The Eilenberg-Mac Lane notion of polynomial functor generalizes to $\f (\gr\op; \rat)$, defined using the free product of groups. (For this and the following material, see \cite[Section 4]{2018arXiv180207574P}, for example.)

\begin{exam}
Let $F$ be a functor on $\gr\op$. Then
\begin{enumerate}
\item 
$F$ has polynomial degree $0$ if and only if it is constant; 
\item 
if $F(0)=0$, then $F$ has polynomial degree one if and only if it is additive, namely, for all $G_1$, $G_2$ in $\gr$, there is a natural isomorphism $F (G_1 \star G_2) \cong F (G_1 ) \oplus F(G_2)$. 
 \end{enumerate}
A concrete example of an additive functor is the dual of the abelianization functor, namely $\mathfrak{a}^\sharp : G \mapsto \hom (G, \rat)$, where $\hom$ denotes group homomorphisms, considering $\rat$ as an abelian group.

More generally, for $d \in \nat$, the functor  $(\mathfrak{a}^\sharp)^{\otimes d}$ has polynomial degree $d$. This splits as a direct sum of simple polynomial functors of polynomial degree exactly $d$  and every such simple polynomial functor  occurs as a composition factor of $(\mathfrak{a}^\sharp)^{\otimes d}$ (possibly with multiplicity $>1$).
\end{exam}

The full subcategory of polynomial functors in $\f (\gr\op; \rat)$ is not semisimple. However, every polynomial functor $F$ (of degree $d$, say)  has a natural increasing filtration of the form 
$$
F_{-1} = 0 \subseteq 
F_0 \subseteq F_1  \subseteq F_2 \subseteq \ldots \subseteq F_d = F,
$$
where each subquotient $F_t/ F_{t-1}$ is semisimple with composition factors of polynomial degree exactly $t$. (The form of this filtration is the reason for working with $\gr \op$  rather than $\gr$.) This justifies the following:

\begin{defn}
\label{defn:analytic_grop}
A functor in $\f (\gr\op; \rat)$  is analytic if it is the union of its polynomial subfunctors. The full subcategory of analytic functors is written $\f_\omega (\gr\op; \rat)$.
\end{defn}

The main result of \cite{MR4835394} gives a convenient model for $\f_\omega (\gr\op; \rat)$. To explain this, we recall  $\cat \lieopd$, the $\rat$-linear category underlying the PROP associated to the Lie operad $\lieopd$. This has set of objects $\nat$ and satisfies the following for $s, t \in \nat$:
$$
\cat \lieopd (s, t) 
\cong  
\left\{
\begin{array}{ll}
0 & s < t \\
\rat \sym_s & s= t \\
\lieopd (s) & t=1.
\end{array}
\right.
$$

In particular, fixing $t$, for any $s \in \nat$, $\cat\lieopd (s, t)$ is a $\rat \sym_s\op$-module so that, using the Schur functor construction, one has the associated analytic functor (on $\fmodq$)
$$
V \mapsto \cat \lieopd (-, t) (V).
$$
Unwinding the construction of the PROP associated to $\lieopd$, this analytic functor identifies as 
$$
V \mapsto \lie(V)^{\otimes t}.
$$ 
Moreover, the natural $\sym_t$-action arising from that on each $\cat \lieopd(s,t)$ corresponds to the place permutation action. This effectively encodes the structure of $\cat \lieopd (-, t)$ via the Schur correspondence. 

The model for analytic functors is given by the following:

\begin{thm}
\label{thm:catlie}
\cite[Theorem 1]{MR4835394}
The category $\f _\omega (\gr\op; \rat)$ is naturally equivalent to the category of $\cat \lieopd$-modules (aka., the category of $\rat$-linear functors from $\cat\lieopd$ to $\rat$-vector spaces).
\end{thm}

This is a useful result; for example, it allows one to read off directly the polynomial filtration of an analytic functor $F$ from its associated $\cat \lieopd$-module, denoted $M_F$ here. In particular, the composition  factors of $F$ of polynomial degree exactly $t$ are in bijective correspondence with those of the $\rat \sym_t$-module $M_F (t)$.

We are interested in the category of {\em outer functors} introduced in \cite{2018arXiv180207574P}:

\begin{defn}
\label{defn:fout}
The  category $\fout (\gr\op; \rat)$ of outer functors is the full subcategory of functors in $\f (\gr\op; \rat)$ on which inner automorphisms of free groups act trivially. 

The category $\fout_\omega (\gr\op; \rat)$ of analytic outer functors is the corresponding full subcategory of $\f_\omega (\gr\op; \rat)$.
\end{defn} 

Outer functors arise in many situations, for example when studying higher Hochschild homology of a wedge of circles as in \cite{2018arXiv180207574P}. For other examples, see \cite{MR4799912} and \cite{MR4830678}.

\begin{exam}
For any $d\in \nat$, $(\mathfrak{a}^\sharp)^{\otimes d}$ is an outer functor. This shows that all simple polynomial functors are outer functors. However, the category of outer functors is not closed under extension, and its structure turns out to be  subtle (see \cite{2018arXiv180207574P,2023arXiv231116881H}, for example).
\end{exam}

 The main result of \cite{MR4696223} provides a counterpart of Theorem \ref{thm:catlie} that models $\fout_\omega (\gr\op; \rat)$. 
To explain this, we review the ingredients. There is a $\rat$-linear functor 
$$
\cat \lieopd 
\twoheadrightarrow 
H_0 \cat \lieopd 
$$
that is the identity on objects and such that $\cat \lieopd (s,t ) \twoheadrightarrow H_0 \cat \lieopd (s,t)$ is surjective; it is an isomorphism for $s=t$. (In the notation of \cite{MR4696223}, $H_0 \cat \lieopd (s, -) = \cat \lieopd (s, -)^\mu$.) The $\rat$-linear category $H_0 \cat \lieopd $ will not be defined explicitly here; the identification of its associated analytic functor in Theorem \ref{thm:identify_H0} below is sufficient for current purposes.

The model for outer analytic functors is the following:

\begin{thm}
\label{thm:catlie_out}
\cite[Theorem 4]{MR4696223}
The category $\fout_\omega (\gr\op; \rat)$ is equivalent to the category of $H_0 \cat \lieopd$-modules. 

Moreover, via the equivalence of Theorem \ref{thm:catlie}, the inclusion $\fout_\omega (\gr\op; \rat) \subset \f (\gr \op; \rat)$ identifies as the restriction along $\cat \lieopd 
\twoheadrightarrow 
H_0 \cat \lieopd $.
\end{thm}

This shows that $H_0 \cat \lieopd$ controls the structure of $\fout_\omega (\gr\op; \rat)$. 

\begin{exam}
\label{exam:H_0_projective}
For $s \in \nat$, the $\cat \lieopd$-module given by $H_0 \cat \lieopd (s, -)$ corresponds under the equivalence of Theorem \ref{thm:catlie_out} to the projective cover in $\fout_\omega (\gr\op; \rat)$ of the functor $(\mathfrak{a}^\sharp)^{\otimes s}$.
\end{exam}

Finally, the relationship with the problem studied in the current paper is explained by the following:

\begin{thm}
\label{thm:identify_H0}
\cite[Theorem 5]{MR4696223}
For $t \in \nat$, the morphism of analytic functors (on $\fmodq$) induced by $\cat \lieopd (-,t) 
\twoheadrightarrow 
H_0 \cat \lieopd (-, t)
$ identifies with the $\rat \sym_t$-equivariant natural (with respect to $V$) projection:
$$
\lie (V)^{\otimes t} 
\twoheadrightarrow 
H_0 (\lie (V); \lie(V)^{\otimes t}).
$$
\end{thm}

For instance, this allows us to recover the underlying bimodules of $H_0 \cat \lieopd$:

\begin{exam}
For $s, t \in \nat$, the $\rat (\sym_s \op \times \sym_t)$-module $H_0 \cat \lieopd (s,t)$ is determined by 
$$
H_0 (\lie (V); \lie(V)^{\otimes t})^{[s]},
$$
where $(-)^{[s]}$ denotes the $s$th homogeneous component of an analytic functor on $\fmodq$.
\end{exam}

\providecommand{\bysame}{\leavevmode\hbox to3em{\hrulefill}\thinspace}
\providecommand{\MR}{\relax\ifhmode\unskip\space\fi MR }
\providecommand{\MRhref}[2]{%
  \href{http://www.ams.org/mathscinet-getitem?mr=#1}{#2}
}
\providecommand{\href}[2]{#2}

\end{document}